\documentclass[11pt]{article}
\usepackage{amsmath,amssymb,bm,amsthm,array,tikz,enumerate}
\usepackage{url}
\usepackage{lscape}
\usepackage[margin=30truemm]{geometry}
\usepackage{float}
\usepackage{makecell}
\usepackage{blkarray}
\usepackage{multirow}
\usepackage{changepage}

\usetikzlibrary{positioning,graphs}

\numberwithin{equation}{section}

\DeclareMathOperator{\Spec}{Spec}

\DeclareMathOperator{\tr}{tr}

\newcommand{\MC}[1]{\mathcal{#1}}
\newcommand{\MB}[1]{\mathbb{#1}}

\newcommand{\BM}[1]{{\bm #1}}

\newcommand{\one}{{\bm 1}}
\newcommand{\mapsfrom}{\mathrel{\reflectbox{$\mapsto$}}}

\newcommand{\TR}{\textcolor{red}}
\newcommand{\TB}{\textcolor{blue}}

\newtheorem{thm}{Theorem}[section]
\newtheorem{lem}[thm]{Lemma}

\newtheorem{cor}[thm]{Corollary}

\newtheorem{claim}{Claim}

\theoremstyle{definition}

\allowdisplaybreaks

\title
{
Spectral characterization of the uniform theta graph $\Theta(t,2)$ and classification of 6-periodic Grover walks
}

\author{
Sho Kubota\thanks{
Department of Mathematics Education,
Aichi University of Education,
1 Hirosawa, Igaya-cho, Kariya, Aichi 448-8542, Japan.
\texttt{skubota@auecc.aichi-edu.ac.jp}}
}
\date{}

\begin{document}
\maketitle

\begin{abstract}
We characterize the uniform theta graph $\Theta(t,2)$ by the spectrum of its normalized adjacency matrix, or equivalently, by the spectrum of its normalized Laplacian matrix.
We also investigate the periodicity of Grover walks on nonregular graphs, which is closely related to the eigenvalues of the normalized adjacency matrix and those of the time evolution matrix of the Grover walk.
We show that the Dutch windmill graph $D_n^{(t)}$ is $2n$-periodic and that the uniform theta graph $\Theta(t,n)$ is $(2n+2)$-periodic.
Furthermore, we completely determine the connected $6$-periodic graphs and prove that they are precisely $D_3^{(t)}$ with $t \geq 2$ and $\Theta(t,2)$ with $t \geq 1$.
\vspace{8pt} 
\end{abstract}

\section{Introduction}

Determining to what extent the spectrum of a matrix associated with a graph determines the structure of the graph is one of the fundamental problems in spectral graph theory.
In particular, the question of whether a graph is uniquely determined, up to isomorphism, by the spectrum of an associated matrix has been widely studied.
The origin of this problem goes back to the work of G\"unthard and Primas in 1956~\cite{gunthard1956zusammenhang}.
Since then, spectral characterization has been studied for various families of graphs.
For background and known results, we refer to the surveys by van Dam and Haemers~\cite{van2003graphs, van2009developments} and the recent survey~\cite{sason2025spectral}.

Let $A$ be the adjacency matrix of a graph $G$, and let $D$ be its degree matrix.
One important matrix associated with a graph in this paper is the normalized adjacency matrix $T(G) := D^{-1/2} A D^{-1/2}$.
From the viewpoint of eigenvalues, this is essentially equivalent to considering the normalized Laplacian matrix $\MC{L}(G) := I - T(G) = I - D^{-1/2}AD^{-1/2}$, where $I$ is the identity matrix.
Although not as extensively studied as for the adjacency matrix, spectral characterization of graphs with respect to the normalized Laplacian matrix has been studied in many papers.
To the best of our knowledge, one of the early studies in this direction is due to van Dam and Omidi~\cite{van2011graphs}.
Since then, the characterization of graphs by their normalized Laplacian spectra has been studied from several different perspectives.
For example, graphs with a small number of distinct normalized Laplacian eigenvalues have been studied in~\cite{braga2015trees, huang2019graphs, li2020split}, while graphs satisfying conditions on the values or multiplicities of specific eigenvalues have been investigated in~\cite{li2014bounds, sun2024characterization, sun2019second, sun2026spectral, tian2021full, zhao2025solving}.
Furthermore, spectral characterization has also been studied for specific graph families, including threshold graphs~\cite{banerjee2017normalized}, generalized friendship graphs~\cite{berman2018family}, and complete multipartite graphs~\cite{sun2020normalized}.
From a different point of view, Butler and Grout~\cite{butler2011construction} computationally investigated nonisomorphic simple graphs with at most nine vertices and counted the graphs having a normalized Laplacian cospectral mate.
Table~1 in \cite{butler2011construction} is particularly interesting because, judging from the computational results for graphs with a small number of vertices, most graphs are determined by the normalized Laplacian spectrum, whereas only a limited number of families of graphs have been theoretically proved to be determined by their normalized Laplacian spectra.


On the other hand, eigenvalue analysis of matrices associated with graphs is also directly useful for the analysis of quantum walks on graphs.
As described in Subsection~\ref{0824-1}, the eigenvalues of the time evolution matrix $U(G)$ of the Grover walk, a discrete-time quantum walk defined only from the structure of the graph, are known to be described in terms of the eigenvalues of the normalized adjacency matrix $T(G)$.
Therefore, problems concerning Grover walks are often reduced to problems in spectral graph theory.
In this paper, we focus on the periodicity of the Grover walk as one such problem.
The periodicity of quantum walks is also expected to have applications in quantum cryptography~\cite{panda2021order, rath2026quantum}.
If there exists a positive integer $\tau$ such that $U(G)^\tau = I$, then the smallest such $\tau$ is called the period of the Grover walk, and $G$ is said to be $\tau$-periodic.
The problem considered here is to determine all graphs $G$ that are $\tau$-periodic for a given positive integer $\tau$.
For graphs that are not necessarily regular, this problem has been solved for $\tau = 2, 4$~\cite{yoshie2017characterizations} and for odd $\tau$~\cite{yoshie2023odd}.
On the other hand, for connected regular graphs, this problem has been studied in~\cite{kubota2025regular}, where it was shown in particular that if $\tau \equiv 2,10 \pmod{12}$, then the only connected $\tau$-periodic regular graph is the cycle graph $C_{\tau}$.

Based on the above background, we present three main results in this paper.
Our first result provides two families of nonregular graphs that induce periodic Grover walks.
Although there have been many studies on the periodicity of Grover walks~\cite{bhakta2024grover, bhakta2025periodicity, higuchi2017periodicity, kubota2022bipartite, yoshie2019periodicities}, most of them have focused on regular graphs.
In this paper, we show that two families of nonregular graphs, the Dutch windmill graphs and the uniform theta graphs, induce periodic Grover walks.
Moreover, their periodicity is proved without using eigenvalue analysis, by directly tracing the time evolution of the Grover walk.
Detailed definitions of the Dutch windmill graphs and the uniform theta graphs are given in Section~\ref{S3}.

\begin{thm}
The Dutch windmill graph $D_n^{(t)}$ is $2n$-periodic for integers $t \geq 2$ and $n \geq 3$, and the uniform theta graph $\Theta(t,n)$ is $(2n+2)$-periodic for positive integers $t$ and $n$.
\end{thm}

The second main result is a complete classification of connected $6$-periodic graphs.
This classification result itself is obtained by combining the characterization by Jost et al.~\cite{jost2023petals} concerning the spectral gap from $1$ of the normalized Laplacian matrix with the spectral mapping theorem for Grover walks.
Therefore, the structural characterization of graphs in this result is based on the result of Jost et al., but the role of this paper is to establish a connection between their result and the periodicity of the Grover walk.
Details are given in Section~\ref{S4}.

\begin{thm}
The only connected graphs that induce a 6-periodic Grover walk are the Dutch windmill graphs $D_3^{(t)}$ with $t\geq 2$ and the uniform theta graphs $\Theta(t,2)$ with $t\geq 1$.
\end{thm}

The third main result is a spectral characterization of the uniform theta graph $\Theta(t,2)$ by its normalized adjacency spectrum.
This result is motivated by the aim of classifying $6$-periodic graphs more directly without using the result of Jost et al.
By the spectral mapping theorem, classifying connected $6$-periodic graphs is reduced to determining connected graphs $G$ satisfying
\[ \Spec(T(G)) \subset \left\{ \pm 1, \pm \frac{1}{2} \right\}. \]
If $G$ is non-bipartite, then the result of Berman et al.~\cite{berman2018family} implies that $G$ is a friendship graph, or equivalently, a Dutch windmill graph $D_3^{(t)}$ with $t \geq 2$ in our notation.
On the other hand, if $G$ is bipartite, then its spectrum has the form
\[ \Spec(T(G)) = \left\{ 1^{(1)}, \frac{1}{2}^{(t)}, -\frac{1}{2}^{(t)}, -1^{(1)} \right\}. \]
In this paper, we characterize the connected graphs with this spectrum and obtain the following result.

\begin{thm}
Let $G = (V, E)$ be a connected graph and let $t \geq 1$ be an integer.
Then, $G$ is isomorphic to the uniform theta graph $\Theta(t,2)$ if and only if
\[ \Spec(T(G)) = \left\{ 1^{(1)}, \frac{1}{2}^{(t)}, -\frac{1}{2}^{(t)}, -1^{(1)} \right\}. \]
\end{thm}

\section{Preliminaries}

See \cite{godsil2013algebraic} for basic terminology related to graphs.
Let $G =(V, E)$ be a graph with vertex set $V$ and edge set $E$.
Throughout this paper, we assume that graphs are simple and finite,
i.e., $|V| < \infty$ and $E \subset \{\{x,y\} \subset V \mid x \neq y\}$.
For a square matrix $M$, we denote by $\Spec(M)$ the multiset of eigenvalues of $M$, counted with algebraic multiplicities.
Define the \emph{adjacency matrix} $A = A(G) \in \MB{C}^{V \times V}$ by
\[ A_{x,y} = \begin{cases}
1 \quad &\text{if $\{x,y\} \in E$,} \\
0 \quad &\text{otherwise.}
\end{cases} \]
In this paper, the multiset of eigenvalues of $A(G)$ is called the \emph{adjacency spectrum} of $G$.
For example, it is well known that the adjacency spectrum of the cycle graph $C_3$ with three vertices is 
\[ \Spec(A(C_3)) = \left \{ 2^{(1)}, -1^{(2)} \right \}. \]
As shown here, the multiplicity of each eigenvalue is indicated by a parenthesized superscript.

The \emph{degree matrix} $D \in \MB{C}^{V \times V}$ of a graph $G=(V,E)$ is defined by $D_{x,y} = d_x\delta_{x,y}$, where $d_x$ denotes the degree of the vertex $x$ and $\delta_{x,y}$ denotes the Kronecker delta.
For a finite set $\Omega$, we denote by $\one_{\Omega}$ the all-ones vector in the vector space $\MB{C}^{\Omega}$.
When the context is clear, we simply write $\one$.

\begin{lem} \label{0821-1}
Let $G = (V, E)$ be a graph, and let $A$ and $D$ denote the adjacency matrix and the degree matrix of $G$, respectively.
Then, we have $A \one_V = D \one_V$.
\end{lem}

\begin{proof}
Indeed, 
\[ (A \one_V)_x
= \sum_{z \in V} A_{x,z} \one_z
= \sum_{z \in V} A_{x,z}
= d_x,
\]
and 
\[ (D \one_V)_x = \sum_{z \in V} D_{x,z} \one_z = \sum_{z \in V} D_{x,z}
= d_x. \qedhere \]
\end{proof}

\subsection{Grover walks}

We introduce the Grover walk, one of the most fundamental discrete-time quantum walks on graphs.
In this paper, we define the time evolution matrix $U$ of the Grover walk by explicitly specifying its entries.
However, as explained later, it is more convenient to regard $U$ as a linear mapping and understand its action by tracking how it maps the standard basis vectors.

Let $G = (V, E)$ be a graph.
Define $\MC{A} = \MC{A}(G)=\{ (x, y), (y, x) \mid \{x, y\} \in E \}$, which is the set of \emph{symmetric arcs} of $G$.
The origin $x$ and terminus $y$ of $a=(x, y) \in \MC{A}$ are denoted by $o(a)$ and $t(a)$, respectively.
We write the inverse arc of $a$ as $a^{-1}$.
Define the \emph{time evolution matrix} $U = U(G) \in \MB{C}^{\MC{A} \times \MC{A}}$ by
\begin{equation} \label{1223-1}
U_{a,b} = \frac{2}{d_{t(b)}} \delta_{o(a), t(b)} - \delta_{a,b^{-1}}.
\end{equation}
The matrix $U$ is unitary.
The discrete-time quantum walk with the time evolution matrix $U$ is called the \emph{Grover walk} on $G$.
Note that Grover walks are also referred to as \emph{arc-reversal walks} or \emph{arc-reversal Grover walks}.
The derivation of the expression for the entries of $U$ in Equality~\eqref{1223-1} can be found in~\cite{kubota2021quantum}.
We regard the vector space $\MB{C}^{\MC{A}}$, whose coordinates are indexed by $\MC{A}$, as a Hilbert space with the standard inner product,
and denote its standard basis by $\{ \BM{e}_a \mid a \in \MC{A} \}$.
We call a vector $\psi \in \MB{C}^{\MC{A}}$ with $\| \psi \| = 1$ a \emph{quantum state}, or simply a \emph{state}.

\begin{lem}[{\cite[Lemmas~2.1 and~2.2]{kubota2026entanglement}}] \label{0507-1}
For $a \in \MC{A}$, we have
\[ U\BM{e}_a = \frac{2}{d_{t(a)}}\sum_{\substack{z \in \MC{A} \\ o(z) = t(a)}} \BM{e}_z - \BM{e}_{a^{-1}}, \]
and
\[ U^* \BM{e}_a = U^{-1} \BM{e}_a = \frac{2}{d_{o(a)}} \sum_{\substack{z \in \MC{A} \\ t(z) = o(a) }} \BM{e}_z - \BM{e}_{a^{-1}}. \]
\end{lem}

Using Lemma~\ref{0507-1}, the time evolution matrix $U$ can be understood visually by regarding it as a linear mapping.
We write the entries of a vector $\Psi \in \MB{C}^{\MC{A}}$
on the arcs of the graph as in Figure~\ref{48}.
If an entry of $\Psi$ is $0$,
we omit the arc itself corresponding to the entry.
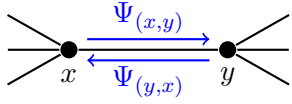
\begin{figure}[ht]
\begin{center}
\begin{tikzpicture}
[scale = 0.7,
line width = 0.8pt,
v/.style = {circle, fill = black, inner sep = 0.8mm},u/.style = {circle, fill = white, inner sep = 0.1mm}]
  \node[u] (1) at (-1.2, 0) {};
  \node[v] (2) at (0, 0) {};
  \node[v] (3) at (3, 0) {};
  \node[u] (4) at (-1.2, 0.68) {};
  \node[u] (5) at (-1.2, -0.68) {};
  \node[u] (7) at (4.2, 0) {};
  \node[u] (8) at (4.2, 0.68) {};
  \node[u] (9) at (4.2, -0.68) {};
  \node[u] (12) at (3.3, 0.2) {};
  \node[u] (13) at (5.7, 0.2) {};
  \draw (0,-0.5) node{$x$};
  \draw (3,-0.5) node{$y$};
  \draw (1) to (2);
  \draw[-] (2) to (4);
  \draw[-] (2) to (3);
  \draw[-] (5) to (2);
  \node[u] (10) at (0.3, 0.2) {};
  \node[u] (11) at (2.7, 0.2) {};
  \draw[draw= blue,->] (10) to (11);
  \node[u] (20) at (0.3, -0.2) {};
  \node[u] (21) at (2.7, -0.2) {};
  \draw[draw= blue,->] (21) to (20);
  \draw[-] (3) to (7);
  \draw[-] (3) to (8);
  \draw[-] (3) to (9);
  \draw (1.5,0.6) node[blue]{$\Psi_{(x,y)}$};
  \draw (1.5,-0.6) node[blue]{$\Psi_{(y,x)}$};
\end{tikzpicture}
\caption{The entries of a vector written on the arcs of the graph} \label{48}
\end{center}
\end{figure}

Figure~\ref{0519-3} illustrates the actions of $U$ and $U^*$.
The middle diagram in Figure~\ref{0519-3} represents the standard basis vector $\BM{e}_{(x,y)}$ corresponding to the arc $(x,y) \in \MC{A}(G)$.
The right diagram represents $U \BM{e}_{(x,y)}$.
We see that the arrow of weight $1$ is transmitted to each non-opposite arrow with weight $\frac{2}{d_y}$, and is reflected onto the opposite arrow with weight $\frac{2}{d_y}-1$.
Similarly, the left diagram represents $U^*\BM{e}_{(x,y)}$.
It shows that the arrow of weight $1$ propagates backward with weight $\frac{2}{d_x}$ along each non-opposite arrow and with weight $\frac{2}{d_x}-1$ along the opposite arrow.

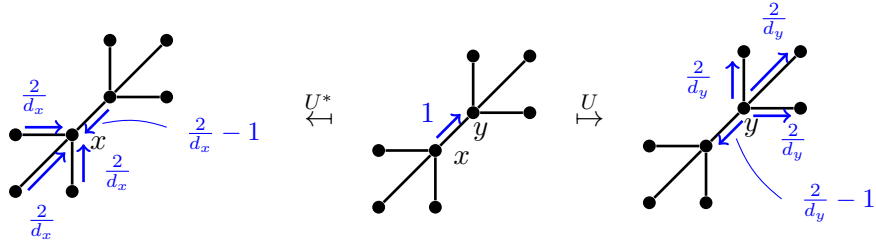
\begin{figure}[htb]
\begin{center}
\begin{tikzpicture}
[scale = 0.5,
v/.style = {circle, fill = black, inner sep = 0.6mm},
u/.style = {circle, fill = white, inner sep = 0.1mm}
]
\node[u] (15) at (3.5, -0.5) {\textcolor{blue}{{\small$\frac{2}{d_x} - 1$}}};
\node[u] (12) at (0.7, -1.5) {\textcolor{blue}{{\small$\frac{2}{d_x}$}}};
\node[u] (13) at (-1.3, -2.8) {\textcolor{blue}{{\small$\frac{2}{d_x}$}}};
\node[u] (14) at (-1.5, 0.5) {\textcolor{blue}{{\small$\frac{2}{d_x}$}}};
\node[u] (x) at (0.2, -0.7) {$x$};
\node[v] (1) at (0.5,0.5) {};
\node[v] (2) at (0.5, 2) {};
\node[v] (3) at (2, 2) {};
\node[v] (4) at (2, 0.5) {};
\draw[line width = 1pt] (1) to (2);
\draw[line width = 1pt] (1) to (3);
\draw[line width = 1pt] (1) to (4);
\node[v] (5) at (-0.5,-0.5) {};
\node[v] (6) at (-0.5, -2) {};
\node[v] (7) at (-2, -2) {};
\node[v] (8) at (-2, -0.5) {};
\draw[line width = 1pt] (1) to (5);
\draw[line width = 1pt] (5) to (6);
\draw[line width = 1pt] (5) to (7);
\draw[line width = 1pt] (5) to (8);
\node[u] (15R) at (0.5, 0.2) {};
\node[u] (15L) at (-0.2, -0.5) {};
\draw[draw = blue, line width = 1pt, ->] (15R) to (15L);
\node[u] (12o) at (-0.2, -0.7) {};
\node[u] (12t) at (-0.2, -1.8) {};
\draw[draw = blue, line width = 1pt, ->] (12t) to (12o);
\node[u] (13o) at (-0.65, -0.93) {};
\node[u] (13t) at (-1.7, -2) {};
\draw[draw = blue, line width = 1pt, ->] (13t) to (13o);
\node[u] (14o) at (-0.7, -0.3) {};
\node[u] (14t) at (-1.8, -0.3) {};
\draw[draw = blue, line width = 1pt, ->] (14t) to (14o);
\draw[-, blue] (2, -0.3) to [bend left = -15] (0.3,-0.2);
\end{tikzpicture}
\raisebox{45pt}{$\quad \overset{U^*}{\mapsfrom} \quad$}
\begin{tikzpicture}
[scale = 0.5,
v/.style = {circle, fill = black, inner sep = 0.6mm},
u/.style = {circle, fill = white, inner sep = 0.1mm}
]
\node[u] (x) at (0.7, 0) {$y$};
\node[u, white] (13) at (1.25, 2.7) {{\scriptsize$\frac{2}{d_y}$}};
\node[u, white] (15) at (0, -2) {{\scriptsize$\frac{2}{d_y} - 1$}};
\node[u] (y) at (0.2, -0.7) {$x$};
\node[u] (99) at (-0.7, 0.5) {\textcolor{blue}{$1$}};
\node[v] (1) at (0.5,0.5) {};
\node[v] (2) at (0.5, 2) {};
\node[v] (3) at (2, 2) {};
\node[v] (4) at (2, 0.5) {};
\draw[line width = 1pt] (1) to (2);
\draw[line width = 1pt] (1) to (3);
\draw[line width = 1pt] (1) to (4);
\node[v] (5) at (-0.5,-0.5) {};
\node[v] (6) at (-0.5, -2) {};
\node[v] (7) at (-2, -2) {};
\node[v] (8) at (-2, -0.5) {};
\draw[line width = 1pt] (1) to (5);
\draw[line width = 1pt] (5) to (6);
\draw[line width = 1pt] (5) to (7);
\draw[line width = 1pt] (5) to (8);
\node[u] (51L) at (-0.5, -0.2) {};
\node[u] (51R) at (0.2, 0.5) {};
\draw[draw = blue, line width = 1pt, ->] (51L) to (51R);
\end{tikzpicture}
\raisebox{45pt}{$\quad \overset{U}{\mapsto} \quad$}
\begin{tikzpicture}
[scale = 0.5,
v/.style = {circle, fill = black, inner sep = 0.6mm},
u/.style = {circle, fill = white, inner sep = 0.1mm}
]
\node[u] (15) at (3, -2) {\textcolor{blue}{{\small$\frac{2}{d_y} - 1$}}};
\node[u] (12) at (-0.7, 1.25) {\textcolor{blue}{{\small$\frac{2}{d_y}$}}};
\node[u] (13) at (1.3, 2.8) {\textcolor{blue}{{\small$\frac{2}{d_y}$}}};
\node[u] (14) at (1.8, -0.4) {\textcolor{blue}{{\small$\frac{2}{d_y}$}}};
\node[u] (x) at (0.7, 0) {$y$};
\node[v] (1) at (0.5,0.5) {};
\node[v] (2) at (0.5, 2) {};
\node[v] (3) at (2, 2) {};
\node[v] (4) at (2, 0.5) {};
\draw[line width = 1pt] (1) to (2);
\draw[line width = 1pt] (1) to (3);
\draw[line width = 1pt] (1) to (4);
\node[v] (5) at (-0.5,-0.5) {};
\node[v] (6) at (-0.5, -2) {};
\node[v] (7) at (-2, -2) {};
\node[v] (8) at (-2, -0.5) {};
\draw[line width = 1pt] (1) to (5);
\draw[line width = 1pt] (5) to (6);
\draw[line width = 1pt] (5) to (7);
\draw[line width = 1pt] (5) to (8);
\node[u] (15R) at (0.5, 0.2) {};
\node[u] (15L) at (-0.2, -0.5) {};
\draw[draw = blue, line width = 1pt, ->] (15R) to (15L);
\node[u] (12o) at (0.2, 0.7) {};
\node[u] (12t) at (0.2, 1.8) {};
\draw[draw = blue, line width = 1pt, ->] (12o) to (12t);
\node[u] (13o) at (0.65, 0.93) {};
\node[u] (13t) at (1.7, 2) {};
\draw[draw = blue, line width = 1pt, ->] (13o) to (13t);
\node[u] (14o) at (0.7, 0.3) {};
\node[u] (14t) at (1.8, 0.3) {};
\draw[draw = blue, line width = 1pt, ->] (14o) to (14t);
\draw[-, blue] (1.5,-1.9) to [bend left = 15] (0.3,-0.3);
\end{tikzpicture}
\caption{The action of $U$ and $U^*$
} \label{0519-3}
\end{center}
\end{figure}

What is important in this paper is the case where the degree of the terminus $t(a)$ of an arc $a$ is 2.
Indeed, let $b$ be the unique arc other than $a^{-1}$ whose origin is $t(a)$.
Then, by Lemma~\ref{0507-1},
\[ U \BM{e}_a = \frac{2}{2}(\BM{e}_{a^{-1}} + \BM{e}_{b}) - \BM{e}_{a^{-1}} = \BM{e}_{b}. \]
Thus, as illustrated in Figure~\ref{44}, the action of the time evolution matrix $U$ on the standard basis vector $\BM{e}_a$ can be viewed as simply moving forward by one step.

\begin{figure}[htb]
\begin{center}
\begin{tikzpicture}
[scale = 0.6,
line width = 0.8pt,
v/.style = {circle, fill = black, inner sep = 0.8mm},
u/.style = {circle, fill = white, inner sep = 0.0mm}]
  \node[u] (a) at (1.5, 0.7) {\textcolor{blue}{{\small $1$}}};
  \node[u] (1) at (-1.2, 0) {};
  \node[v] (2) at (0, 0) {};
  \node[v] (3) at (3, 0) {};
  \node[u] (4) at (-1.2, 0.68) {};
  \node[u] (5) at (-1.2, -0.68) {};
  \node[v] (6) at (6, 0) {};
  \node[u] (7) at (7.2, 0) {};
  \node[u] (8) at (7.2, 0.68) {};
  \node[u] (9) at (7.2, -0.68) {};
  \node[u] (10) at (0.3, 0.2) {};
  \node[u] (11) at (2.7, 0.2) {};
  \node[u] (12) at (3.3, 0.2) {};
  \node[u] (13) at (5.7, 0.2) {};
  \draw (1) to (2);
  \draw[-] (2) to (4);
  \draw[-] (2) to (3);
  \draw[-] (5) to (2);
  \draw[draw= blue,->] (10) to (11);
  \draw[-] (3) to (6);
  \draw[-] (6) to (7);
  \draw[-] (6) to (8);
  \draw[-] (6) to (9);
\end{tikzpicture}
\raisebox{3.5mm}{$\quad \overset{U}{\mapsto} \quad$}
\begin{tikzpicture}
[scale = 0.7,
line width = 0.8pt,
v/.style = {circle, fill = black, inner sep = 0.8mm},u/.style = {circle, fill = white, inner sep = 0.1mm}]
  \node[u] (a) at (4.5, 0.7) {\textcolor{blue}{{\small $1$}}};
  \node[u] (1) at (-1.2, 0) {};
  \node[v] (2) at (0, 0) {};
  \node[v] (3) at (3, 0) {};
  \node[u] (4) at (-1.2, 0.68) {};
  \node[u] (5) at (-1.2, -0.68) {};
  \node[v] (6) at (6, 0) {};
  \node[u] (7) at (7.2, 0) {};
  \node[u] (8) at (7.2, 0.68) {};
  \node[u] (9) at (7.2, -0.68) {};
  \node[u] (10) at (0.3, 0.2) {};
  \node[u] (11) at (2.7, 0.2) {};
  \node[u] (12) at (3.3, 0.2) {};
  \node[u] (13) at (5.7, 0.2) {};
  \draw (4.5,0.7) node[blue]{};
  \draw (1) to (2);
  \draw[-] (2) to (4);
  \draw[-] (2) to (3);
  \draw[-] (5) to (2);
  \draw[draw= blue,->] (12) to (13);
  \draw[-] (3) to (6);
  \draw[-] (6) to (7);
  \draw[-] (6) to (8);
  \draw[-] (6) to (9);
\end{tikzpicture}
\caption{The action of $U$ in the case where the degree of the terminus of an arc is $2$} \label{44}
\end{center}
\end{figure}
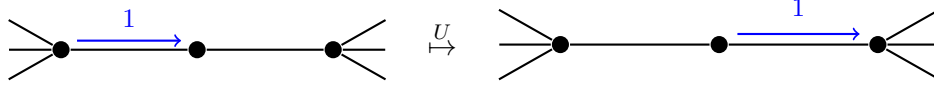

Let $G$ be a graph, and let $U := U(G)$ be its time evolution matrix.
If there exists a positive integer $\tau$ such that $U^{\tau} = I$, then we say that $G$ is \emph{periodic}, or that $G$ \emph{induces a periodic Grover walk}.
If $G$ is periodic, then the smallest positive integer $\tau$ satisfying $U^{\tau} = I$ is called the \emph{period} of $G$, and $G$ is said to be \emph{$\tau$-periodic}.
The Grover walk on $G$ is also said to be \emph{$\tau$-periodic}.
For example, it follows from Figure~\ref{44} that the cycle graph $C_n$ on $n$ vertices is periodic and that its period is $n$.
Indeed, since the cycle graph is $2$-regular, it is easy to verify that $U(C_n)^n \BM{e}_a = \BM{e}_a$ for every standard basis vector $\BM{e}_a \in \MB{C}^{\MC{A}(C_n)}$, and hence $U(C_n)^n = I$.
For recent developments on the periodicity of Grover walks, we refer to~\cite{bhakta2026state} and the references therein.
At present, studies on the periodicity of Grover walks have mainly focused on regular graphs.
For nonregular graphs, apart from complete bipartite graphs, generalized Bethe trees are the only family of graphs whose periodicity has been studied~\cite{kubota2018generalizedBT}.

\subsection{Spectral mapping theorem for Grover walks} \label{0824-1}

In analyzing the Grover walk on a graph, not only the behavior of weighted arcs described above but also the spectral analysis of the time evolution matrix $U$ is important.
Although the matrix $U$ given by Equality~\eqref{1223-1} appears complicated at first glance, its eigenvalues are in fact known to be described in terms of the eigenvalues of the normalized adjacency matrix.

Let $G=(V,E)$ be a graph, and let $A$ and $D$ denote its adjacency matrix and degree matrix, respectively.
Define the matrix $T = T(G) \in \MB{C}^{V\times V}$ by
\[ T := D^{-1/2} A D^{-1/2}, \]
which is called the \emph{normalized adjacency matrix}.
In the context of quantum walks, the normalized adjacency matrix is also called the \emph{discriminant matrix}.
The following theorem is known as the \emph{spectral mapping theorem for the Grover walk}.

\begin{thm}[\cite{higuchi2014spectral}] \label{0820-1}
Let $G = (V, E)$ be a connected graph.
Then we have
\[ \Spec(U(G)) = \{ e^{\pm i \arccos \lambda} \mid \lambda \in \Spec(T(G)) \} \cup \{ 1^{(M_1)}, -1^{(M_{-1})} \}, \]
where $M_1 = |E| - |V| + 1$ and $M_{-1} = |E| - |V| + \dim \ker (T(G)+I)$.
\end{thm}

\section{Dutch windmill graphs and uniform theta graphs} \label{S3}

The \emph{Dutch windmill graph}, denoted by $D_n^{(t)}$, is obtained from $t$ copies of the cycle graph $C_n$ by identifying one vertex from each copy.
Figure~\ref{0810-1} depicts $D_4^{(3)}$.
When $n=3$, this graph is known by various names,
including the petal graph, the friendship graph, and the windmill graph.
Note that a graph obtained by identifying one vertex from each of several complete graphs, not cycle graphs, is called a \emph{windmill graph}
and is also referred to as a \emph{generalized friendship graph} in some literature.

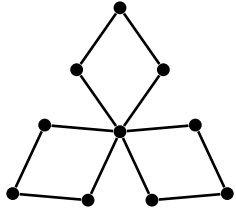
\begin{figure}[htb]
\centering
\begin{tikzpicture}
[scale = 0.5,
v/.style = {circle, fill = black, inner sep = 0.6mm},
u/.style = {circle, fill = white, inner sep = 0.1mm}
]
\node[v] (0) at (0,0) {};
\node[v] (11) at (55:2) {};
\node[v] (12) at (90:{4*cos(35)}) {};
\node[v] (13) at (125:2) {};
\node[v] (21) at (175:2) {};
\node[v] (22) at (210:{4*cos(35)}) {};
\node[v] (23) at (245:2) {};
\node[v] (31) at (295:2) {};
\node[v] (32) at (330:{4*cos(35)}) {};
\node[v] (33) at (5:2) {};
\draw[line width = 1pt] (0) -- (11) -- (12) -- (13) -- (0);
\draw[line width = 1pt] (0) -- (21) -- (22) -- (23) -- (0);
\draw[line width = 1pt] (0) -- (31) -- (32) -- (33) -- (0);
\end{tikzpicture}
\caption{The Dutch windmill graph $D_4^{(3)}$} \label{0810-1}
\end{figure}

\begin{thm} \label{0722-2}
Let $t \geq 2$ and $n \geq 3$ be integers.
The Dutch windmill graph $D_n^{(t)}$ is $2n$-periodic.
In particular, $D_3^{(t)}$ is $6$-periodic.
\end{thm}

\begin{proof}
Let $U := U(D_n^{(t)})$.
We show that every standard basis vector $\BM{e}_a \in \MB{C}^{\MC{A}(D_n^{(t)})}$ satisfies $U^{2n}\BM{e}_a=\BM{e}_a$.
Let $x$ be the central vertex of the Dutch windmill graph $D_n^{(t)}$.
First, we show that an arbitrary arc $a$ with $t(a)=x$ satisfies $U^{2n}\BM{e}_a=\BM{e}_a$.
By symmetry, we may assume that the initial state is as shown in Figure~\ref{0811-1} without loss of generality.
By repeatedly applying Lemma~\ref{0507-1}, this initial state returns to itself after $2n$ steps as shown in Figure~\ref{0811-2}.

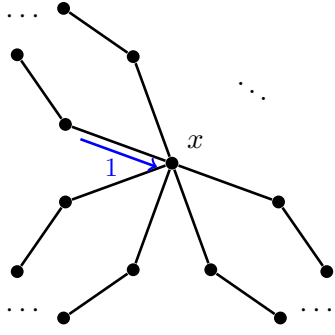
\begin{figure}[htb]
\centering
\begin{tikzpicture}
[scale = 0.5,
v/.style = {circle, fill = black, inner sep = 0.6mm},
u/.style = {circle, fill = white, inner sep = 0.0mm}
]
\node[u] (w) at (-1.6, -0.1) {{\small \TB{$1$}}};
\node[v, label=above right:$x$] (0) at (0,0) {};
\node[v] (11) at (110:3) {};
\node[v] (12) at (125:5) {};
\node[u] (13) at (135:5.5) {$\cdots$};
\node[v] (14) at (145:5) {};
\node[v] (15) at (160:3) {};
\draw[line width=1pt] (0) -- (11) -- (12);
\draw[line width=1pt] (14) -- (15) -- (0);
\node[v] (21) at (200:3) {};
\node[v] (22) at (215:5) {};
\node[u] (23) at (225:5.5) {$\cdots$};
\node[v] (24) at (235:5) {};
\node[v] (25) at (250:3) {};
\draw[line width=1pt] (0) -- (21) -- (22);
\draw[line width=1pt] (24) -- (25) -- (0);
\node[v] (31) at (290:3) {};
\node[v] (32) at (305:5) {};
\node[u] (33) at (315:5.5) {$\cdots$};
\node[v] (34) at (325:5) {};
\node[v] (35) at (340:3) {};
\draw[line width=1pt] (0) -- (31) -- (32);
\draw[line width=1pt] (34) -- (35) -- (0);
\node[u] at (45:3) {$\ddots$};
\draw[draw=blue, line width=1pt, ->, shorten <=6pt, shorten >=6pt]
([yshift=-7pt]15.center) -- ([yshift=-7pt]0.center);
\end{tikzpicture}
\caption{The standard basis vector $\BM{e}_a$ with $t(a) = x$} \label{0811-1}
\end{figure}

Next, we consider the standard basis vector $\BM{e}_a$ corresponding to an arbitrary arc $a \in \MC{A}(D_{n}^{(t)})$.
Among $\BM{e}_a, U\BM{e}_a, \dots, U^{n-1}\BM{e}_a$, there exists a standard basis vector corresponding to an arc whose terminal vertex is $x$.
Suppose that $U^j \BM{e}_a$ is such a vector.
By the preceding argument, we have $U^{2n} U^j \BM{e}_a = U^j \BM{e}_a$, and hence $U^{2n}\BM{e}_a=\BM{e}_a$.
Therefore, the Dutch windmill graph $D_n^{(t)}$ is $2n$-periodic.
\end{proof}

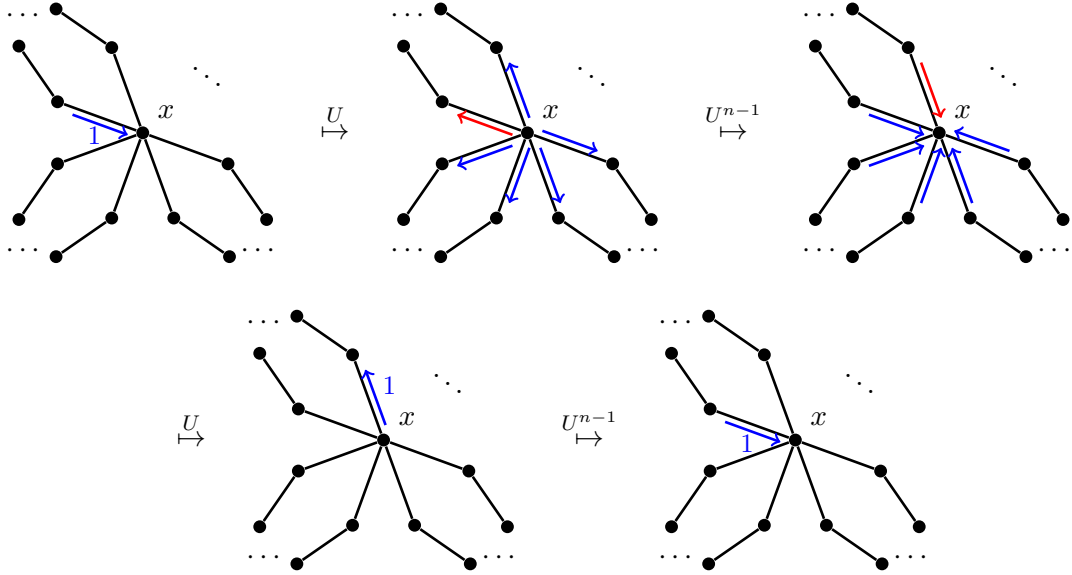
\begin{figure}[htb]
\centering
\begin{tikzpicture}
[scale = 0.4,
v/.style = {circle, fill = black, inner sep = 0.6mm},
u/.style = {circle, fill = white, inner sep = 0.1mm}
]
\node[u] (w) at (-1.6, -0.1) {{\small \TB{$1$}}};
\node[v, label=above right:$x$] (0) at (0,0) {};
\node[v] (11) at (110:3) {};
\node[v] (12) at (125:5) {};
\node[u] (13) at (135:5.5) {$\cdots$};
\node[v] (14) at (145:5) {};
\node[v] (15) at (160:3) {};
\draw[line width=1pt] (0) -- (11) -- (12);
\draw[line width=1pt] (14) -- (15) -- (0);
\node[v] (21) at (200:3) {};
\node[v] (22) at (215:5) {};
\node[u] (23) at (225:5.5) {$\cdots$};
\node[v] (24) at (235:5) {};
\node[v] (25) at (250:3) {};
\draw[line width=1pt] (0) -- (21) -- (22);
\draw[line width=1pt] (24) -- (25) -- (0);
\node[v] (31) at (290:3) {};
\node[v] (32) at (305:5) {};
\node[u] (33) at (315:5.5) {$\cdots$};
\node[v] (34) at (325:5) {};
\node[v] (35) at (340:3) {};
\draw[line width=1pt] (0) -- (31) -- (32);
\draw[line width=1pt] (34) -- (35) -- (0);
\node[u] at (45:3) {$\ddots$};
\draw[draw=blue, line width=1pt, ->, shorten <=6pt, shorten >=6pt]
([yshift=-7pt]15.center) -- ([yshift=-7pt]0.center);
\end{tikzpicture}
\raisebox{50pt}{$\quad \overset{U}{\mapsto} \quad$}
\begin{tikzpicture}
[scale = 0.4,
v/.style = {circle, fill = black, inner sep = 0.6mm},
u/.style = {circle, fill = white, inner sep = 0.1mm}
]
\node[v, label=above right:$x$] (0) at (0,0) {};
\node[v] (11) at (110:3) {};
\node[v] (12) at (125:5) {};
\node[u] (13) at (135:5.5) {$\cdots$};
\node[v] (14) at (145:5) {};
\node[v] (15) at (160:3) {};
\draw[line width=1pt] (0) -- (11) -- (12);
\draw[line width=1pt] (14) -- (15) -- (0);
\node[v] (21) at (200:3) {};
\node[v] (22) at (215:5) {};
\node[u] (23) at (225:5.5) {$\cdots$};
\node[v] (24) at (235:5) {};
\node[v] (25) at (250:3) {};
\draw[line width=1pt] (0) -- (21) -- (22);
\draw[line width=1pt] (24) -- (25) -- (0);
\node[v] (31) at (290:3) {};
\node[v] (32) at (305:5) {};
\node[u] (33) at (315:5.5) {$\cdots$};
\node[v] (34) at (325:5) {};
\node[v] (35) at (340:3) {};
\draw[line width=1pt] (0) -- (31) -- (32);
\draw[line width=1pt] (34) -- (35) -- (0);
\node[u] at (45:3) {$\ddots$};
\draw[draw=red, line width=1pt, <-, shorten <=6pt, shorten >=6pt]
([yshift=-7pt]15.center) -- ([yshift=-7pt]0.center);
\draw[draw=blue, line width=1pt, ->, shorten <=6pt, shorten >=6pt]
([yshift=-7pt]0.center) -- ([yshift=-7pt]21.center);
\draw[draw=blue, line width=1pt, ->, shorten <=6pt, shorten >=6pt]
([xshift=7pt]0.center) -- ([xshift=7pt]25.center);
\draw[draw=blue, line width=1pt, ->, shorten <=6pt, shorten >=6pt]
([xshift=7pt]0.center) -- ([xshift=7pt]31.center);
\draw[draw=blue, line width=1pt, ->, shorten <=6pt, shorten >=6pt]
([yshift=7pt]0.center) -- ([yshift=7pt]35.center);
\draw[draw=blue, line width=1pt, ->, shorten <=6pt, shorten >=6pt]
([xshift=7pt]0.center) -- ([xshift=7pt]11.center);
\end{tikzpicture}
\raisebox{50pt}{$\quad \overset{U^{n-1}}{\mapsto} \quad$}
\begin{tikzpicture}
[scale = 0.4,
v/.style = {circle, fill = black, inner sep = 0.6mm},
u/.style = {circle, fill = white, inner sep = 0.1mm}
]
\node[v, label=above right:$x$] (0) at (0,0) {};
\node[v] (11) at (110:3) {};
\node[v] (12) at (125:5) {};
\node[u] (13) at (135:5.5) {$\cdots$};
\node[v] (14) at (145:5) {};
\node[v] (15) at (160:3) {};
\draw[line width=1pt] (0) -- (11) -- (12);
\draw[line width=1pt] (14) -- (15) -- (0);
\node[v] (21) at (200:3) {};
\node[v] (22) at (215:5) {};
\node[u] (23) at (225:5.5) {$\cdots$};
\node[v] (24) at (235:5) {};
\node[v] (25) at (250:3) {};
\draw[line width=1pt] (0) -- (21) -- (22);
\draw[line width=1pt] (24) -- (25) -- (0);
\node[v] (31) at (290:3) {};
\node[v] (32) at (305:5) {};
\node[u] (33) at (315:5.5) {$\cdots$};
\node[v] (34) at (325:5) {};
\node[v] (35) at (340:3) {};
\draw[line width=1pt] (0) -- (31) -- (32);
\draw[line width=1pt] (34) -- (35) -- (0);
\node[u] at (45:3) {$\ddots$};
\draw[draw=blue, line width=1pt, ->, shorten <=6pt, shorten >=6pt]
([yshift=-7pt]15.center) -- ([yshift=-7pt]0.center);
\draw[draw=blue, line width=1pt, <-, shorten <=6pt, shorten >=6pt]
([yshift=-7pt]0.center) -- ([yshift=-7pt]21.center);
\draw[draw=blue, line width=1pt, <-, shorten <=6pt, shorten >=6pt]
([xshift=7pt]0.center) -- ([xshift=7pt]25.center);
\draw[draw=blue, line width=1pt, <-, shorten <=6pt, shorten >=6pt]
([xshift=7pt]0.center) -- ([xshift=7pt]31.center);
\draw[draw=blue, line width=1pt, <-, shorten <=6pt, shorten >=6pt]
([yshift=7pt]0.center) -- ([yshift=7pt]35.center);
\draw[draw=red, line width=1pt, <-, shorten <=6pt, shorten >=6pt]
([xshift=7pt]0.center) -- ([xshift=7pt]11.center);
\end{tikzpicture} \\[10pt]
\raisebox{50pt}{$\quad \overset{U}{\mapsto} \quad$}
\begin{tikzpicture}
[scale = 0.4,
v/.style = {circle, fill = black, inner sep = 0.6mm},
u/.style = {circle, fill = white, inner sep = 0.1mm}
]
\node[u] (w) at (0.2, 1.8) {{\small \TB{$1$}}};
\node[v, label=above right:$x$] (0) at (0,0) {};
\node[v] (11) at (110:3) {};
\node[v] (12) at (125:5) {};
\node[u] (13) at (135:5.5) {$\cdots$};
\node[v] (14) at (145:5) {};
\node[v] (15) at (160:3) {};
\draw[line width=1pt] (0) -- (11) -- (12);
\draw[line width=1pt] (14) -- (15) -- (0);
\node[v] (21) at (200:3) {};
\node[v] (22) at (215:5) {};
\node[u] (23) at (225:5.5) {$\cdots$};
\node[v] (24) at (235:5) {};
\node[v] (25) at (250:3) {};
\draw[line width=1pt] (0) -- (21) -- (22);
\draw[line width=1pt] (24) -- (25) -- (0);
\node[v] (31) at (290:3) {};
\node[v] (32) at (305:5) {};
\node[u] (33) at (315:5.5) {$\cdots$};
\node[v] (34) at (325:5) {};
\node[v] (35) at (340:3) {};
\draw[line width=1pt] (0) -- (31) -- (32);
\draw[line width=1pt] (34) -- (35) -- (0);
\node[u] at (45:3) {$\ddots$};
\draw[draw=blue, line width=1pt, ->, shorten <=6pt, shorten >=6pt]
([xshift=7pt]0.center) -- ([xshift=7pt]11.center);
\end{tikzpicture}
\raisebox{50pt}{$\quad \overset{U^{n-1}}{\mapsto} \quad$}
\begin{tikzpicture}
[scale = 0.4,
v/.style = {circle, fill = black, inner sep = 0.6mm},
u/.style = {circle, fill = white, inner sep = 0.1mm}
]
\node[u] (w) at (-1.6, -0.1) {{\small \TB{$1$}}};
\node[v, label=above right:$x$] (0) at (0,0) {};
\node[v] (11) at (110:3) {};
\node[v] (12) at (125:5) {};
\node[u] (13) at (135:5.5) {$\cdots$};
\node[v] (14) at (145:5) {};
\node[v] (15) at (160:3) {};
\draw[line width=1pt] (0) -- (11) -- (12);
\draw[line width=1pt] (14) -- (15) -- (0);
\node[v] (21) at (200:3) {};
\node[v] (22) at (215:5) {};
\node[u] (23) at (225:5.5) {$\cdots$};
\node[v] (24) at (235:5) {};
\node[v] (25) at (250:3) {};
\draw[line width=1pt] (0) -- (21) -- (22);
\draw[line width=1pt] (24) -- (25) -- (0);
\node[v] (31) at (290:3) {};
\node[v] (32) at (305:5) {};
\node[u] (33) at (315:5.5) {$\cdots$};
\node[v] (34) at (325:5) {};
\node[v] (35) at (340:3) {};
\draw[line width=1pt] (0) -- (31) -- (32);
\draw[line width=1pt] (34) -- (35) -- (0);
\node[u] at (45:3) {$\ddots$};
\draw[draw=blue, line width=1pt, ->, shorten <=6pt, shorten >=6pt]
([yshift=-7pt]15.center) -- ([yshift=-7pt]0.center);
\end{tikzpicture}
\caption{Proof of $U^{2n}\BM{e}_a=\BM{e}_a$.
For arrows whose weights are not indicated, blue arrows have weight $\frac{1}{t}$, while red arrows have weight $\frac{1}{t}-1$.} \label{0811-2}
\end{figure}

Note that the condition $t\geq 2$ in the theorem above is essential, because $D_n^{(1)}$ is isomorphic to the cycle graph $C_n$, which has period $n$.

\begin{cor} \label{0901-1}
Let $t \geq 2$ be an integer.
Then, we have
\[ \Spec(T(D_3^{(t)})) = \left\{ 1^{(1)}, \frac{1}{2}^{(t-1)}, -\frac{1}{2}^{(t+1)} \right\}. \]
\end{cor}

\begin{proof}
By Theorem~\ref{0722-2}, we have $U(D_3^{(t)})^6 = I$, and hence every eigenvalue of $U(D_3^{(t)})$ belongs to $\{ e^{\frac{2k \pi}{6}i} \mid  k = 0,1,2, \dots, 5 \}$.
Furthermore, Theorem~\ref{0820-1} implies that every eigenvalue of $T(D_3^{(t)})$ belongs to
\[ \left \{ \cos \frac{2k \pi}{6} \, \middle| \, k = 0,1,2, \dots, 5 \right \} = \left \{ \pm 1, \pm \frac{1}{2} \right \}. \]
Since $D_3^{(t)}$ is non-bipartite, we have $ -1 \notin \Spec(T(D_3^{(t)}))$.
Thus, we may write
\[ \Spec(T(D_3^{(t)})) = \left \{ 1^{(1)}, \frac{1}{2}^{(a)}, -\frac{1}{2}^{(b)} \right \} \]
for nonnegative integers $a,b$.
Since $D_3^{(t)}$ has $2t+1$ vertices and $\tr(T(D_3^{(t)})) = 0$, we obtain
\[ 2t+1 = 1+a+b, \qquad 1 + \frac{a}{2} - \frac{b}{2} = 0. \]
Solving this system, we obtain $(a, b) = (t-1, t+1)$, which proves the statement.
\end{proof}

The \emph{uniform theta graph}, denoted by $\Theta(t,n)$,
is defined as follows.
First, take $t$ pairwise vertex-disjoint path graphs on $n$ vertices and add two new vertices.
For each path graph, join one endpoint to the first new vertex and the other endpoint to the second new vertex.
The graph obtained in this way is the uniform theta graph $\Theta(t,n)$.

Figure~\ref{0810-2} depicts $\Theta(3,4)$.
Note that $\Theta(t,2)$ is precisely the graph called a \emph{book graph} by Jost et al.~\cite{jost2023petals}.
Besides $\Theta(t,2)$, uniform theta graphs include several familiar graphs as special cases.
In particular, $\Theta(t,1)$ is isomorphic to the complete bipartite graph $K_{2,t}$,
$\Theta(2,n)$ is isomorphic to the cycle graph $C_{2n+2}$,
and $\Theta(1,n)$ is isomorphic to the path graph $P_{n+2}$.

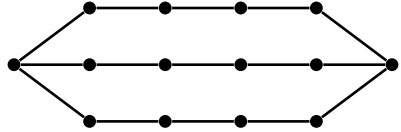
\begin{figure}[htb]
\centering
\begin{tikzpicture}
[scale = 0.5,
v/.style = {circle, fill = black, inner sep = 0.6mm},
u/.style = {circle, fill = white, inner sep = 0.0mm}
]
\node[v] (L) at (-5, 0) {};
\node[v] (11) at (-3, 1.5) {};
\node[v] (21) at (-1, 1.5) {};
\node[v] (31) at (1, 1.5) {};
\node[v] (41) at (3, 1.5) {};
\node[v] (12) at (-3, 0) {};
\node[v] (22) at (-1, 0) {};
\node[v] (32) at (1, 0) {};
\node[v] (42) at (3, 0) {};
\node[v] (13) at (-3, -1.5) {};
\node[v] (23) at (-1, -1.5) {};
\node[v] (33) at (1, -1.5) {};
\node[v] (43) at (3, -1.5) {};
\node[v] (R) at (5, 0) {};
\draw[line width = 1pt] (L) -- (11) -- (21) -- (31) -- (41) -- (R);
\draw[line width = 1pt] (L) -- (12) -- (22) -- (32) -- (42) -- (R);
\draw[line width = 1pt] (L) -- (13) -- (23) -- (33) -- (43) -- (R);
\end{tikzpicture}
\caption{The uniform theta graph $\Theta(3, 4)$} \label{0810-2}
\end{figure}

\begin{thm} \label{0722-3}
Let $t$ and $n$ be positive integers.
The uniform theta graph $\Theta(t,n)$ is $(2n+2)$-periodic.
In particular, $\Theta(t,2)$ is $6$-periodic.
\end{thm}

\begin{proof}
When $t = 1, 2$, we have $\Theta(1,n) = P_{n+2}$ and $\Theta(2,n) = C_{2n+2}$, respectively, and in both cases the graph is $(2n+2)$-periodic~\cite[Theorems~6.2 and~7.1]{kubota2021periodicity}.
Hence, in what follows, we consider the case $t\geq 3$.
Let $U := U(\Theta(t,n))$.
We show that every standard basis vector $\BM{e}_a \in \MB{C}^{\MC{A}(\Theta(t,n))}$ satisfies $U^{2n+2}\BM{e}_a=\BM{e}_a$.
The graph $\Theta(t,n)$ has exactly two vertices of degree $t$. Let the left one be $L$ and the right one be $R$.
First, we show that an arbitrary arc $a$ with $t(a)=L$ satisfies $U^{2n+2}\BM{e}_a=\BM{e}_a$.
By symmetry, we may assume that the initial state is as shown in Figure~\ref{0812-1} without loss of generality.
By repeatedly applying Lemma~\ref{0507-1}, this initial state returns to itself after $2n+2$ steps as shown in Figure~\ref{0812-2}.

\begin{figure}[htb]
\centering
\begin{tikzpicture}
[scale = 0.5,
v/.style = {circle, fill = black, inner sep = 0.6mm},
u/.style = {circle, fill = white, inner sep = 0.0mm}
]
\node[u] (w) at (-4.5, 1.4) {{\small \TB{$1$}}};
\node[u] at (-2, -0.5) {$\vdots$};
\node[u] at (2, -0.5) {$\vdots$};
\node[u] at (0, 2) {$\cdots$};
\node[u] at (0, 0.5) {$\cdots$};
\node[u] at (0, -2) {$\cdots$};
\node[v, label=left:$L$] (L) at (-5, 0) {};
\node[v] (11) at (-3, 2) {};
\node[v] (21) at (-1, 2) {};
\node[v] (31) at (1, 2) {};
\node[v] (41) at (3, 2) {};
\node[v] (12) at (-3, 0.5) {};
\node[v] (22) at (-1, 0.5) {};
\node[v] (32) at (1, 0.5) {};
\node[v] (42) at (3, 0.5) {};
\node[v] (13) at (-3, -2) {};
\node[v] (23) at (-1, -2) {};
\node[v] (33) at (1, -2) {};
\node[v] (43) at (3, -2) {};
\node[v, label=right:$R$] (R) at (5, 0) {};
\draw[line width = 1pt] (L) -- (11) -- (21);
\draw[line width = 1pt] (31) -- (41) -- (R);
\draw[line width = 1pt] (L) -- (12) -- (22);
\draw[line width = 1pt] (32) -- (42) -- (R);
\draw[line width = 1pt] (L) -- (13) -- (23);
\draw[line width = 1pt] (33) -- (43) -- (R);
%
\draw[draw=blue, line width=1pt, ->, shorten <=6pt, shorten >=6pt]
([xshift=-7pt]11.center) -- ([xshift=-7pt]L.center);
\end{tikzpicture}
\caption{The standard basis vector $\BM{e}_a$ with $t(a) = L$} \label{0812-1}
\end{figure}
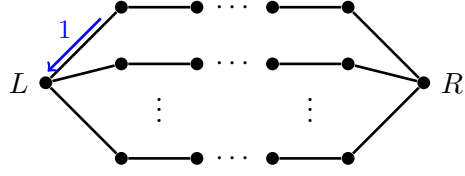

Next, we consider the standard basis vector $\BM{e}_a$ corresponding to an arbitrary arc $a \in \MC{A}(\Theta(t,n))$.
Among $\BM{e}_a, U\BM{e}_a, \dots, U^{n+1}\BM{e}_a$, there exists a standard basis vector corresponding to an arc whose terminal vertex is $L$ or $R$.
Without loss of generality, we may assume that $U^j\BM{e}_a$ is the standard basis vector corresponding to an arc whose terminal vertex is $L$.
By the preceding argument, we have $U^{2n+2} U^j \BM{e}_a = U^j \BM{e}_a$, and hence $U^{2n+2}\BM{e}_a=\BM{e}_a$.
Therefore, the uniform theta graph $\Theta(t,n)$ is $(2n+2)$-periodic.
\end{proof}


\begin{figure}[htb]
\centering
\begin{tikzpicture}
[scale = 0.5,
v/.style = {circle, fill = black, inner sep = 0.6mm},
u/.style = {circle, fill = white, inner sep = 0.0mm}
]
\node[u] (w) at (-4.5, 1.4) {{\small \TB{$1$}}};
\node[u] at (-2, -0.5) {$\vdots$};
\node[u] at (2, -0.5) {$\vdots$};
\node[u] at (0, 2) {$\cdots$};
\node[u] at (0, 0.5) {$\cdots$};
\node[u] at (0, -2) {$\cdots$};
\node[v, label=left:$L$] (L) at (-5, 0) {};
\node[v] (11) at (-3, 2) {};
\node[v] (21) at (-1, 2) {};
\node[v] (31) at (1, 2) {};
\node[v] (41) at (3, 2) {};
\node[v] (12) at (-3, 0.5) {};
\node[v] (22) at (-1, 0.5) {};
\node[v] (32) at (1, 0.5) {};
\node[v] (42) at (3, 0.5) {};
\node[v] (13) at (-3, -2) {};
\node[v] (23) at (-1, -2) {};
\node[v] (33) at (1, -2) {};
\node[v] (43) at (3, -2) {};
\node[v, label=right:$R$] (R) at (5, 0) {};
\draw[line width = 1pt] (L) -- (11) -- (21);
\draw[line width = 1pt] (31) -- (41) -- (R);
\draw[line width = 1pt] (L) -- (12) -- (22);
\draw[line width = 1pt] (32) -- (42) -- (R);
\draw[line width = 1pt] (L) -- (13) -- (23);
\draw[line width = 1pt] (33) -- (43) -- (R);
%
\draw[draw=blue, line width=1pt, ->, shorten <=6pt, shorten >=6pt]
([xshift=-7pt]11.center) -- ([xshift=-7pt]L.center);
\end{tikzpicture}
\raisebox{31pt}{$\quad \overset{U}{\mapsto} \quad$}
\begin{tikzpicture}
[scale = 0.5,
v/.style = {circle, fill = black, inner sep = 0.6mm},
u/.style = {circle, fill = white, inner sep = 0.0mm}
]
\node[u] at (-2, -0.5) {$\vdots$};
\node[u] at (2, -0.5) {$\vdots$};
\node[u] at (0, 2) {$\cdots$};
\node[u] at (0, 0.5) {$\cdots$};
\node[u] at (0, -2) {$\cdots$};
\node[v, label=left:$L$] (L) at (-5, 0) {};
\node[v] (11) at (-3, 2) {};
\node[v] (21) at (-1, 2) {};
\node[v] (31) at (1, 2) {};
\node[v] (41) at (3, 2) {};
\node[v] (12) at (-3, 0.5) {};
\node[v] (22) at (-1, 0.5) {};
\node[v] (32) at (1, 0.5) {};
\node[v] (42) at (3, 0.5) {};
\node[v] (13) at (-3, -2) {};
\node[v] (23) at (-1, -2) {};
\node[v] (33) at (1, -2) {};
\node[v] (43) at (3, -2) {};
\node[v, label=right:$R$] (R) at (5, 0) {};
\draw[line width = 1pt] (L) -- (11) -- (21);
\draw[line width = 1pt] (31) -- (41) -- (R);
\draw[line width = 1pt] (L) -- (12) -- (22);
\draw[line width = 1pt] (32) -- (42) -- (R);
\draw[line width = 1pt] (L) -- (13) -- (23);
\draw[line width = 1pt] (33) -- (43) -- (R);
%
\draw[draw=red, line width=1pt, <-, shorten <=6pt, shorten >=6pt]
([yshift=7pt]11.center) -- ([yshift=7pt]L.center);
\draw[draw=blue, line width=1pt, ->, shorten <=6pt, shorten >=6pt]
([yshift=-7pt]L.center) -- ([yshift=-7pt]12.center);
\draw[draw=blue, line width=1pt, ->, shorten <=6pt, shorten >=6pt]
([yshift=-7pt]L.center) -- ([yshift=-7pt]13.center);
\end{tikzpicture} \\[10pt]
\raisebox{31pt}{$\quad \overset{U^n}{\mapsto} \quad$}
\begin{tikzpicture}
[scale = 0.5,
v/.style = {circle, fill = black, inner sep = 0.6mm},
u/.style = {circle, fill = white, inner sep = 0.0mm}
]
\node[u] at (-2, -0.5) {$\vdots$};
\node[u] at (2, -0.5) {$\vdots$};
\node[u] at (0, 2) {$\cdots$};
\node[u] at (0, 0.5) {$\cdots$};
\node[u] at (0, -2) {$\cdots$};
\node[v, label=left:$L$] (L) at (-5, 0) {};
\node[v] (11) at (-3, 2) {};
\node[v] (21) at (-1, 2) {};
\node[v] (31) at (1, 2) {};
\node[v] (41) at (3, 2) {};
\node[v] (12) at (-3, 0.5) {};
\node[v] (22) at (-1, 0.5) {};
\node[v] (32) at (1, 0.5) {};
\node[v] (42) at (3, 0.5) {};
\node[v] (13) at (-3, -2) {};
\node[v] (23) at (-1, -2) {};
\node[v] (33) at (1, -2) {};
\node[v] (43) at (3, -2) {};
\node[v, label=right:$R$] (R) at (5, 0) {};
\draw[line width = 1pt] (L) -- (11) -- (21);
\draw[line width = 1pt] (31) -- (41) -- (R);
\draw[line width = 1pt] (L) -- (12) -- (22);
\draw[line width = 1pt] (32) -- (42) -- (R);
\draw[line width = 1pt] (L) -- (13) -- (23);
\draw[line width = 1pt] (33) -- (43) -- (R);
%
\draw[draw=red, line width=1pt, ->, shorten <=6pt, shorten >=6pt]
([xshift=7pt]41.center) -- ([xshift=7pt]R.center);
\draw[draw=blue, line width=1pt, ->, shorten <=6pt, shorten >=6pt]
([yshift=-7pt]42.center) -- ([yshift=-7pt]R.center);
\draw[draw=blue, line width=1pt, ->, shorten <=6pt, shorten >=6pt]
([yshift=-7pt]43.center) -- ([yshift=-7pt]R.center);
\end{tikzpicture}
\raisebox{31pt}{$\quad \overset{U}{\mapsto} \quad$}
\begin{tikzpicture}
[scale = 0.5,
v/.style = {circle, fill = black, inner sep = 0.6mm},
u/.style = {circle, fill = white, inner sep = 0.0mm}
]
\node[u] (w) at (4.5, 1.4) {{\small \TB{$1$}}};
\node[u] at (-2, -0.5) {$\vdots$};
\node[u] at (2, -0.5) {$\vdots$};
\node[u] at (0, 2) {$\cdots$};
\node[u] at (0, 0.5) {$\cdots$};
\node[u] at (0, -2) {$\cdots$};
\node[v, label=left:$L$] (L) at (-5, 0) {};
\node[v] (11) at (-3, 2) {};
\node[v] (21) at (-1, 2) {};
\node[v] (31) at (1, 2) {};
\node[v] (41) at (3, 2) {};
\node[v] (12) at (-3, 0.5) {};
\node[v] (22) at (-1, 0.5) {};
\node[v] (32) at (1, 0.5) {};
\node[v] (42) at (3, 0.5) {};
\node[v] (13) at (-3, -2) {};
\node[v] (23) at (-1, -2) {};
\node[v] (33) at (1, -2) {};
\node[v] (43) at (3, -2) {};
\node[v, label=right:$R$] (R) at (5, 0) {};
\draw[line width = 1pt] (L) -- (11) -- (21);
\draw[line width = 1pt] (31) -- (41) -- (R);
\draw[line width = 1pt] (L) -- (12) -- (22);
\draw[line width = 1pt] (32) -- (42) -- (R);
\draw[line width = 1pt] (L) -- (13) -- (23);
\draw[line width = 1pt] (33) -- (43) -- (R);
%
\draw[draw=blue, line width=1pt, <-, shorten <=6pt, shorten >=6pt]
([yshift=7pt]41.center) -- ([yshift=7pt]R.center);
\end{tikzpicture} \\[10pt]
\raisebox{31pt}{$\quad \overset{U^n}{\mapsto} \quad$}
\begin{tikzpicture}
[scale = 0.5,
v/.style = {circle, fill = black, inner sep = 0.6mm},
u/.style = {circle, fill = white, inner sep = 0.0mm}
]
\node[u] (w) at (-4.5, 1.4) {{\small \TB{$1$}}};
\node[u] at (-2, -0.5) {$\vdots$};
\node[u] at (2, -0.5) {$\vdots$};
\node[u] at (0, 2) {$\cdots$};
\node[u] at (0, 0.5) {$\cdots$};
\node[u] at (0, -2) {$\cdots$};
\node[v, label=left:$L$] (L) at (-5, 0) {};
\node[v] (11) at (-3, 2) {};
\node[v] (21) at (-1, 2) {};
\node[v] (31) at (1, 2) {};
\node[v] (41) at (3, 2) {};
\node[v] (12) at (-3, 0.5) {};
\node[v] (22) at (-1, 0.5) {};
\node[v] (32) at (1, 0.5) {};
\node[v] (42) at (3, 0.5) {};
\node[v] (13) at (-3, -2) {};
\node[v] (23) at (-1, -2) {};
\node[v] (33) at (1, -2) {};
\node[v] (43) at (3, -2) {};
\node[v, label=right:$R$] (R) at (5, 0) {};
\draw[line width = 1pt] (L) -- (11) -- (21);
\draw[line width = 1pt] (31) -- (41) -- (R);
\draw[line width = 1pt] (L) -- (12) -- (22);
\draw[line width = 1pt] (32) -- (42) -- (R);
\draw[line width = 1pt] (L) -- (13) -- (23);
\draw[line width = 1pt] (33) -- (43) -- (R);
%
\draw[draw=blue, line width=1pt, ->, shorten <=6pt, shorten >=6pt]
([xshift=-7pt]11.center) -- ([xshift=-7pt]L.center);
\end{tikzpicture}
\caption{Proof of $U^{2n+2}\BM{e}_a=\BM{e}_a$.
For arrows whose weights are not indicated, blue arrows have weight $\frac{2}{t}$, while red arrows have weight $\frac{2}{t}-1$.} \label{0812-2}
\end{figure}
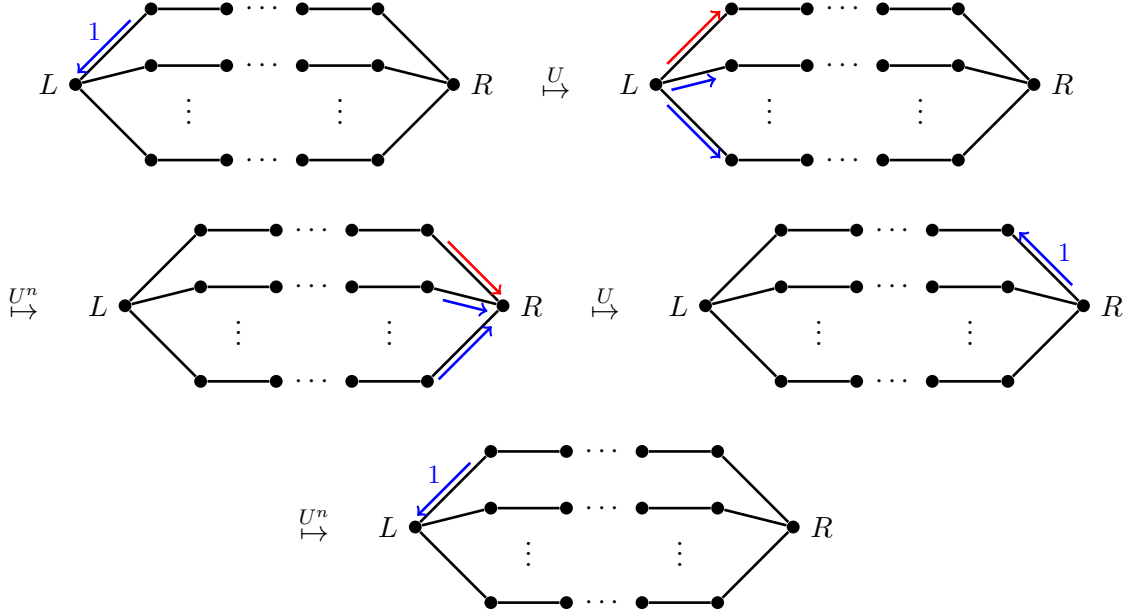


\begin{cor} \label{0820-2}
We have
\[ \Spec(T(\Theta(t,2))) = \left\{ 1^{(1)}, \frac{1}{2}^{(t)}, -\frac{1}{2}^{(t)}, -1^{(1)} \right\}. \]
\end{cor}

\begin{proof}
By Theorem~\ref{0722-3}, we have $U(\Theta(t,2))^6 = I$, and hence every eigenvalue of $U(\Theta(t,2))$ belongs to $\{ e^{\frac{2k \pi}{6}i} \mid  k = 0,1,2, \dots, 5 \}$.
Furthermore, Theorem~\ref{0820-1} implies that every eigenvalue of $T(\Theta(t,2))$ belongs to
\[ \left \{ \cos \frac{2k \pi}{6} \, \middle| \, k = 0,1,2, \dots, 5 \right \} = \left \{ \pm 1, \pm \frac{1}{2} \right \}. \]
Since $\Theta(t,2)$ is bipartite, the eigenvalues of its normalized adjacency matrix are symmetric about $0$.
Thus, we may write
\[ \Spec(T(\Theta(t,2))) = \left \{ 1^{(1)}, \frac{1}{2}^{(a)}, -\frac{1}{2}^{(a)}, -1^{(1)} \right \} \]
for some nonnegative integer $a$.
Since $\Theta(t,2)$ has $2t+2$ vertices, we have $2t+2 = 2a+2$, and hence $a=t$.
\end{proof}

In the next section, we prove that the only connected graphs that induce a $6$-periodic Grover walk are the Dutch windmill graphs $D_3^{(t)}$ and the uniform theta graphs $\Theta(t,2)$.

\section{Classification of $6$-periodic graphs} \label{S4}

Let $G=(V,E)$ be a graph.
The matrix defined by $\MC{L}(G) = \MC{L} := I-T\in\MB{C}^{V \times V}$ is called the \emph{normalized Laplacian matrix} of $G$.
Since the spectrum of the normalized adjacency matrix $T$ is contained in $[-1,1]$, the spectrum of the normalized Laplacian matrix $\MC{L}$ is contained in $[0,2]$.
Spectral gaps of the normalized Laplacian matrix are usually measured from the endpoints $0$ and $2$.
On the other hand, Jost et al.~\cite{jost2023petals} defined the \emph{spectral gap from $1$} by
\[ \varepsilon := \min_{\lambda \in \Spec(\MC{L}(G))} |1 - \lambda|. \]
Remarkably, Jost et al.\ completely characterized the graphs that attain the upper bound for the spectral gap from $1$:

\begin{thm}[{\cite[Theorem~1]{jost2023petals}}] \label{0722-1}
For any connected graph $G = (V, E)$ with $|V| \geq 3$,
\[ \varepsilon \leq \frac{1}{2}. \]
Moreover, equality is achieved if and only if $G$ is either a petal graph $D_3^{(t)}$ or a book graph $\Theta(t, 2)$.
\end{thm}

The graphs that induce a 6-periodic Grover walk are essentially determined by the result of Jost et al.

\begin{thm}
The only connected graphs that induce a 6-periodic Grover walk are the Dutch windmill graphs $D_3^{(t)}$ with $t\geq 2$ and the uniform theta graphs $\Theta(t,2)$ with $t\geq 1$.
\end{thm}

\begin{proof}
Let $G = (V, E)$ be a connected $6$-periodic graph.
First, since the only connected graph on two vertices,
namely $P_2$, is $2$-periodic~\cite{kubota2021periodicity}, we may assume that $|V|\geq 3$.
The spectrum of the time evolution matrix of $G$ must satisfy
$\Spec(U(G)) \subset \{e^{\frac{2k \pi}{6}i} \mid k = 0, \dots ,5\}$,
and hence Theorem~\ref{0820-1} implies that
\[ \Spec(T(G)) \subset \left\{ \cos \frac{2k}{6} \pi \, \middle | \, k = 0, \dots, 5 \right\} = \left\{\pm 1, \pm \frac{1}{2} \right\}. \]
Since $G$ is connected, $T(G)$ has $1$ as a simple eigenvalue.
Moreover, by the Perron--Frobenius theorem, the multiplicity of the eigenvalue $-1$ is at most one.
Since $|V|\geq 3$ and $\tr T(G)=0$,
it follows that $-\frac{1}{2}$ must be an eigenvalue of $T(G)$.
This implies that
\[ \left\{ 0, \frac{3}{2} \right\} \subset \Spec(\MC{L}(G)) \subset \left\{ 0, \frac{1}{2}, \frac{3}{2}, 2 \right\}. \]
In particular, the spectral gap from $1$ of the graph $G$ is $\varepsilon = \frac{1}{2}$.
Therefore, by Theorem~\ref{0722-1},
the graph $G$ must be either $D_3^{(t)}$ or $\Theta(t,2)$.
Conversely, the graphs $D_3^{(t)}$ and $\Theta(t,2)$ are 6-periodic by Theorems~\ref{0722-2} and~\ref{0722-3}, except for $D_3^{(1)}$.
\end{proof}

\section{A more direct characterization of graphs with normalized adjacency eigenvalues in $\{ \pm 1, \pm \frac{1}{2}\}$}

In the previous section, we showed that the only connected graphs that induce a 6-periodic Grover walk are the Dutch windmill graphs $D_3^{(t)}$ and the uniform theta graphs $\Theta(t,2)$.
The proof relies essentially on the result of Jost et al. concerning the spectral gap from $1$ of the normalized Laplacian matrix.
However, since the result of Jost et al.\ is established in a highly general setting, using it to determine $6$-periodic graphs is somewhat indirect.
Therefore, in this section, we give a more direct proof of this characterization based on the eigenvalues of the normalized adjacency matrix $T$.

Let $G=(V,E)$ be a connected graph that is 6-periodic,
and suppose that $|V| \geq 3$.
Then the eigenvalues of its normalized adjacency matrix $T = T(G)$ must satisfy
\[ \Spec(T) \subset \left \{\pm1,\pm\frac{1}{2} \right\}. \]

First, consider the case where $-1 \notin \Spec(T)$.
Let the multiplicities of the eigenvalues be
\[ \Spec(T) = \left\{ 1^{(1)}, \frac{1}{2}^{(a)},
-\frac{1}{2}^{(b)} \right\} \]
for nonnegative integers $a, b$.
Since the sum of the eigenvalues of $T$ is zero, it follows that $b=a+2$.
Therefore, the spectrum of the normalized Laplacian matrix $\MC{L} = \MC{L}(G)$ is given by
\[  \Spec(\MC{L}) = \left\{ 0^{(1)}, \frac{1}{2}^{(a)}, \frac{3}{2}^{(a+2)} \right\}. \]
By Corollary~\ref{0901-1}, this is the same spectrum as that of the normalized Laplacian matrix of the friendship graph, denoted by $D_3^{(a+1)}$ in this paper.
Graphs with such a normalized Laplacian spectrum were characterized by Berman et al.~\cite{berman2018family}, and $G$ is isomorphic to the friendship graph $D_3^{(a+1)}$.

Next, we consider the case $-1 \in \Spec(T)$.
In this case, since the graph $G$ is bipartite,
we may write
\[ \Spec(T) = \left\{ 1^{(1)}, \frac{1}{2}^{(a)}, -\frac{1}{2}^{(a)}, -1^{(1)} \right\} \]
for a positive integer $a$.
We will show that such a graph must be the uniform theta graph $\Theta(a,2)$.

\begin{lem} \label{0729-1}
Let $G = (V, E)$ be a connected bipartite graph with partite sets $V_1$ and $V_2$, and let $m = |E|$.
For $i = 1,2$, let $D_i$ denote the block of the degree matrix corresponding to the partite set $V_i$.
That is,
\[ D= \begin{bmatrix} D_1 & O \\ O & D_2 \end{bmatrix}. \]
Then, the normalized eigenvector of $T$ corresponding to the eigenvalue $1$ is 
\[ \frac{1}{\sqrt{2m}}D^{1/2} \one_V. \]
Similarly, the normalized eigenvector of $T$ corresponding to the eigenvalue $-1$ is
\begin{equation} \label{0727-1}
\frac{1}{\sqrt{2m}} \begin{bmatrix} D_1^{1/2} \one_{V_1} \\ -D_2^{1/2} \one_{V_2} \end{bmatrix}
\end{equation}
\end{lem}

\begin{proof}
Since it is straightforward to verify the statement for the eigenvalue $1$ of $T$, we only show that the vector given in~\eqref{0727-1} is the normalized eigenvector of $T$ corresponding to the eigenvalue $-1$.
Since $G$ is bipartite, the adjacency matrix $A$ can be written as
\[ A= \begin{bmatrix} O & N \\ N^{\top} & O \end{bmatrix} \]
with the vertices ordered according to the partite sets $V_1$ and $V_2$.
Under this representation, Lemma~\ref{0821-1} yields
\begin{equation} \label{0727-3}
\begin{bmatrix} N \one_{V_2} \\ N^{\top} \one_{V_1} \end{bmatrix}
= \begin{bmatrix} O & N \\ N^{\top} & O \end{bmatrix} \begin{bmatrix} \one_{V_1} \\ \one_{V_2} \end{bmatrix}
= \begin{bmatrix} D_1 & O \\ O & D_2 \end{bmatrix} \begin{bmatrix} \one_{V_1} \\ \one_{V_2} \end{bmatrix}
= \begin{bmatrix} D_1 \one_{V_1} \\ D_2 \one_{V_2} \end{bmatrix}
\end{equation}
With these preparations, we prove the claim.
We have
\begin{align*}
T \begin{bmatrix} D_1^{1/2} \one_{V_1} \\ -D_2^{1/2} \one_{V_2} \end{bmatrix}
&= \begin{bmatrix} D_1^{-1/2} & O \\ O & D_2^{-1/2} \end{bmatrix}
\begin{bmatrix} O & N \\ N^{\top} & O \end{bmatrix}
\begin{bmatrix} D_1^{-1/2} & O \\ O & D_2^{-1/2} \end{bmatrix} \begin{bmatrix} D_1^{1/2} \one_{V_1} \\ -D_2^{1/2} \one_{V_2} \end{bmatrix} \\
&= \begin{bmatrix} -D_1^{-1/2}N \one_{V_2} \\ D_2^{-1/2}N^{\top} \one_{V_1} \end{bmatrix} \\
&= \begin{bmatrix} -D_1^{-1/2} D_1 \one_{V_1} \\ D_2^{-1/2} D_2 \one_{V_2} \end{bmatrix} \tag{by Equality~\eqref{0727-3}} \\
&= - \begin{bmatrix} D_1^{1/2} \one_{V_1} \\ -D_2^{1/2} \one_{V_2} \end{bmatrix}
\end{align*}
Moreover, by the handshaking lemma, the norm is
\[ \left\| \begin{bmatrix} D_1^{1/2} \one_{V_1} \\ -D_2^{1/2} \one_{V_2} \end{bmatrix} \right\| = \sqrt{\sum_{x \in V} \sqrt{d_x}^2} = \sqrt{2m}, \]
and hence the claim follows.
\end{proof}

The following theorem is obtained by adapting the well-known argument concerning the Hoffman polynomial~\cite{hoffman1963polynomial} for adjacency matrices of connected regular graphs to normalized adjacency matrices of connected bipartite graphs.

\begin{thm}
Let $G = (V, E)$ be a connected bipartite graph with partite sets $V_1$ and $V_2$, and let $m = |E|$.
For $i = 1,2$, let $D_i$ denote the block of the degree matrix corresponding to the partite set $V_i$.
Assume that the spectrum of the normalized adjacency matrix of $G$ is 
\[ \Spec(T(G))=\left\{ \pm 1^{(1)},\pm \lambda_1^{(m_1)},\dots,\pm \lambda_t^{(m_t)} \right \}. \]
Define the polynomial $f$ by $f(x)=(x^2-\lambda_1^2)(x^2-\lambda_2^2)\cdots(x^2-\lambda_t^2)$.
Then, we have
\[ f(T)= \frac{f(1)}{m}
\begin{bmatrix} D_1^{1/2}J_1D_1^{1/2} & O \\ O & D_2^{1/2}J_2D_2^{1/2} \end{bmatrix}, \]
where $J_i=\one_{V_i} \one_{V_i}^{\top}$ is the all-ones matrix for $i = 1, 2$.
\end{thm}

\begin{proof}
Let $n$ be the number of vertices of $G$.
Since $f(1)=f(-1)$ and $f(\pm \lambda_i)=0$ for every $i$, the spectrum of $f(T)$ is
\[ \Spec(f(T))= \left \{ f(1)^{(2)},0^{(n-2)} \right\}. \]
Since every eigenvector of $T$ is also an eigenvector of $f(T)$, denoting by $P$ the projection matrix of $f(T)$ corresponding to the eigenvalue $f(1)$, we obtain from Lemma~\ref{0729-1} that
\begin{align*}
P &= \left( \frac{1}{\sqrt{2m}} \begin{bmatrix} D_1^{1/2} \one_{V_1} \\ D_2^{1/2} \one_{V_2} \end{bmatrix} \right)
\left( \frac{1}{\sqrt{2m}} \begin{bmatrix} D_1^{1/2} \one_{V_1} \\ D_2^{1/2} \one_{V_2} \end{bmatrix} \right)^{\top} \\
&\hspace{2em} + \left( \frac{1}{\sqrt{2m}} \begin{bmatrix} D_1^{1/2} \one_{V_1} \\ -D_2^{1/2} \one_{V_2} \end{bmatrix} \right)
\left( \frac{1}{\sqrt{2m}} \begin{bmatrix} D_1^{1/2} \one_{V_1} \\ -D_2^{1/2} \one_{V_2} \end{bmatrix} \right)^{\top} \\
&= \frac{1}{m}
\begin{bmatrix} D_1^{1/2}J_1D_1^{1/2} & O \\ O & D_2^{1/2}J_2D_2^{1/2} \end{bmatrix}.
\end{align*}
Hence, the spectral decomposition of $f(T)$ is
\[ f(T)=f(1)P = \frac{f(1)}{m}
\begin{bmatrix} D_1^{1/2}J_1D_1^{1/2} & O \\ O & D_2^{1/2}J_2D_2^{1/2} \end{bmatrix} \]
and the statement of the theorem follows.
\end{proof}

\begin{cor} \label{0803-1}
Let $G = (V, E)$ be a connected bipartite graph with partite sets $V_1$ and $V_2$, and let $m = |E|$.
For $i = 1,2$, let $D_i$ denote the block of the degree matrix corresponding to the partite set $V_i$.
If the spectrum of the normalized adjacency matrix of $G$ is 
\[ \Spec(T(G))=\left\{ \pm 1^{(1)}, \pm \frac{1}{2}^{(a)}  \right \} \]
for a positive integer $a$, then we have
\[ T^2 = \frac{1}{4} I_V + \frac{3}{4m}
\begin{bmatrix} D_1^{1/2}J_1D_1^{1/2} & O \\ O & D_2^{1/2}J_2D_2^{1/2} \end{bmatrix}.
\]
\end{cor}

Note that
\[ (D_i^{1/2} J_i D_i^{1/2})_{x,y} = \sqrt{d_x d_y} \]
for $x, y \in V_i$ and $i = 1,2$.

On the other hand, the entries of $T^2$ have a combinatorial interpretation.
Combining this with Corollary~\ref{0803-1} yields the following fundamental relation between the spectrum and the structural properties of a graph.
For a vertex $x$ of a graph $G$, let $N(x)$ denote the set of vertices adjacent to $x$.

\begin{cor} \label{0803-2}
Let $G = (V, E)$ be a connected bipartite graph, and let $m = |E|$.
If the spectrum of the normalized adjacency matrix of $G$ is 
\[ \Spec(T(G))=\left\{ \pm 1^{(1)}, \pm \frac{1}{2}^{(a)}  \right \} \]
for some positive integer $a$, then we have the following.
\begin{enumerate}[(i)]
\item For any vertex $x\in V$, we have
\[ \sum_{z \in N(x)} \frac{1}{d_z} = \frac{d_x}{4} + \frac{3d_x^2}{4m}. \]
\item For any distinct vertices $x, y \in V$ belonging to the same partite set, we have
\[ \sum_{z \in N(x) \cap N(y)} \frac{1}{d_z} = \frac{3 d_x d_y}{4m}. \]
\end{enumerate}
\end{cor}

\begin{proof}
Let $T = T(G)$.
For any vertices $x, y \in V$, we have
\[ (T^2)_{x,y}
= \sum_{z \in V} T_{x,z} T_{z,y}
= \sum_{z \in V} \frac{1}{\sqrt{d_x d_z}} A_{x,z} \frac{1}{\sqrt{d_z d_y}} A_{z,y}
= \frac{1}{\sqrt{d_x d_y}} \sum_{z \in N(x) \cap N(y)} \frac{1}{d_z}.
\]
If $x=y$, Corollary~\ref{0803-1} gives
\[ \frac{1}{d_x} \sum_{z \in N(x)} \frac{1}{d_z} = \frac{1}{4} + \frac{3}{4m} d_x.  \]
Multiplying both sides by $d_x$, we obtain~(i).
Also, if the two vertices $x$ and $y$ are distinct but belong to the same partite set, Corollary~\ref{0803-1} gives
\[ \frac{1}{\sqrt{d_x d_y}} \sum_{z \in N(x) \cap N(y)} \frac{1}{d_z}
= \frac{3}{4m} \sqrt{d_x d_y}, \]
and multiplying both sides by $\sqrt{d_xd_y}$ yields~(ii).
\end{proof}

Relying essentially only on Corollary~\ref{0803-2}, we now give a spectral characterization of the uniform theta graph $\Theta(t,2)$.

\begin{thm} \label{0903-2}
Let $G = (V, E)$ be a connected graph and let $t \geq 1$ be an integer.
Then, $G$ is isomorphic to the uniform theta graph $\Theta(t,2)$ if and only if
\[ \Spec(T(G)) = \left\{ 1^{(1)}, \frac{1}{2}^{(t)}, -\frac{1}{2}^{(t)}, -1^{(1)} \right\}. \]
\end{thm}


\begin{proof}
The necessity has already been shown in Corollary~\ref{0820-2}.
We prove the converse.
Let $n = |V|$, $m = |E|$, and let $\delta$ be the minimum degree of $G$.
\begin{adjustwidth}{1em}{1em}
\begin{claim}
We have $\delta \leq 2$.
\end{claim}
\begin{proof}[Proof of Claim]
\renewcommand{\qedsymbol}{$\blacklozenge$}
First, we show that $\delta \leq 3$, and then that $\delta \neq 3$.
Let $T = T(G)$.
For any vertex $x \in V$, we have
\[ (T^2)_{x,x} = \frac{1}{d_x} \sum_{z \in N(x)} \frac{1}{d_z}
\leq \frac{1}{d_x} \sum_{z \in N(x)} \frac{1}{\delta} = \frac{1}{\delta}. \]
Hence,
\begin{equation} \label{0804-1}
\tr(T^2) = \sum_{x \in V} (T^2)_{x,x}
\leq \sum_{x \in V} \frac{1}{\delta} = \frac{n}{\delta}.
\end{equation}
On the other hand, since the trace of $T^2$ is equal to the sum of its eigenvalues, we have
\begin{equation} \label{0804-2}
\tr(T^2) = 1^2 + t \left( \frac{1}{2} \right)^2 + t \left( - \frac{1}{2} \right)^2 + (-1)^2 = \frac{2t+8}{4} = \frac{n+6}{4}.
\end{equation}
Inequality~\eqref{0804-1} and Equality~\eqref{0804-2} yield
\[ \frac{n+6}{4} \leq \frac{n}{\delta}, \]
and hence
\[ \delta \leq \frac{4n}{n+6} = 4 - \frac{24}{n+6} < 4. \]
Thus, we have $\delta \leq 3$.

Next, we assume that $\delta = 3$ and derive a contradiction.
Let $u$ be a vertex of minimum degree, and write $N(u) = \{ p_1, p_2, p_3 \}$.
Without loss of generality, we may assume that $d_{p_1} \leq d_{p_2} \leq d_{p_3}$.
Applying Corollary~\ref{0803-2}(i) to the vertex $u$, we obtain
\begin{align}
\frac{3}{4} + \frac{3 \cdot 3^2}{4m}
&= \frac{1}{d_{p_1}} + \frac{1}{d_{p_2}} + \frac{1}{d_{p_3}} \label{0805-1} \\
&\leq \frac{3}{d_{p_1}}, \notag
\end{align}
and hence
\[ d_{p_1} \leq \frac{4m}{m+9} = 4 - \frac{36}{m+9} < 4. \]
By the assumption $\delta = 3$, we must have $d_{p_1} = 3$.
Next, applying Corollary~\ref{0803-2}(ii) to the vertices $p_1$ and $p_2$, we obtain
\[ \sum_{z \in N(p_1) \cap N(p_2)} \frac{1}{d_z}
= \frac{3 d_{p_1} d_{p_2}}{4m} = \frac{9d_{p_2}}{4m}. \]
On the other hand, since 
\[ \sum_{z \in N(p_1) \cap N(p_2)} \frac{1}{d_z} \geq \frac{1}{d_u} = \frac{1}{3}, \]
this inequality yields
\[ \frac{1}{d_{p_2}} \leq \frac{27}{4m}. \]
This, together with Equality~\eqref{0805-1}, yields
\begin{equation} \label{0805-2}
\frac{1}{3} + \frac{1}{d_{p_2}} + \frac{1}{d_{p_3}} =
\frac{3}{4} + \frac{27}{4m} \geq \frac{3}{4} + \frac{1}{d_{p_2}}.
\end{equation}
On the other hand, since $d_{p_2} \leq d_{p_3}$, we have
\begin{equation} \label{0805-3}
\frac{1}{3} + \frac{1}{d_{p_2}} + \frac{1}{d_{p_3}}
\leq \frac{1}{3} + \frac{1}{d_{p_2}} + \frac{1}{d_{p_2}}
= \frac{1}{3} + \frac{2}{d_{p_2}}.
\end{equation}
Inequalities~\eqref{0805-2} and~\eqref{0805-3} yield $d_{p_2} \leq \frac{12}{5} < 3$, contradicting the assumption that the minimum degree $\delta$ is $3$.
\end{proof}
\end{adjustwidth}
First, we deal with the case $\delta = 1$.
Let $u$ be a vertex of minimum degree, and let $p$ be its unique neighbor.
Applying Corollary~\ref{0803-2}(i) to the vertex $u$, we obtain
\[ \frac{1}{d_p} = \frac{1}{4} + \frac{3}{4m}, \]
that is, $(d_p - 4)(m+3) = -12$.
Noting that $m+3>0$ and $d_p-4>-4$, the solutions of this equation are 
\[ (d_p, m) = (1,1), (2,3), (3,9). \]
If $(d_p, m) = (1,1)$, then $G$ is isomorphic to the complete graph $K_2$, contradicting $t \geq 1$.
If $(d_p, m) = (2,3)$, then $G$ is isomorphic to the path graph $P_4$, which is precisely the uniform theta graph $\Theta(1,2)$.
Next, we consider the case $(d_p,m)=(3,9)$.
Let $N(p) \setminus \{ u \} = \{q_1, q_2\}$.
For each $i \in \{ 1,2 \}$, applying Corollary~\ref{0803-2}(ii) to the vertices $u$ and $q_i$ yields
\[ \frac{1}{d_p} = \frac{3d_{q_i} d_u}{4m}. \]
Since $d_p = 3$ and $m = 9$ in the present case, it follows that $d_{q_i} = 4$.
Let $N(q_1)\setminus \{p\} = \{r_1, r_2, r_3\}$.
Applying Corollary~\ref{0803-2}(i) to the vertex $q_1$, we obtain
\[ \frac{1}{3} + \frac{1}{d_{r_1}} + \frac{1}{d_{r_2}} + \frac{1}{d_{r_3}} = \frac{4}{4} + \frac{3 \cdot 4^2}{4 \cdot 9},
\]
that is,
\begin{equation} \label{0903-1}
\frac{1}{d_{r_1}} + \frac{1}{d_{r_2}} + \frac{1}{d_{r_3}} = 2.
\end{equation}
Suppose that $d_{r_i} = 1$ for some $i \in \{1,2,3\}$.
Applying Corollary~\ref{0803-2}(i) to the vertex $r_i$, we have
\[ \frac{1}{4} = \frac{1}{d_{q_1}} = \frac{d_{r_i}}{4}+\frac{3d_{r_i}^2}{4m} = \frac{1}{3},
\]
a contradiction.
Therefore, $d_{r_i} \geq 2$ for every $i \in \{1,2,3\}$.
However, this implies that
\[ \frac{1}{d_{r_1}}+\frac{1}{d_{r_2}}+\frac{1}{d_{r_3}} \leq \frac{1}{2}+\frac{1}{2}+\frac{1}{2} = \frac{3}{2},
\]
which contradicts Equality~\eqref{0903-1}.
Therefore, $(d_p,m) \neq (3,9)$.
Summarizing the argument so far, we have shown that if $\delta = 1$, then $G$ is isomorphic to $\Theta(1,2)$.

In what follows, we deal with the case $\delta = 2$.
As before, let $u$ be a vertex of minimum degree, and write $N(u)=\{ p_1, p_2 \}$.
Without loss of generality, we may assume that $d_{p_1} \leq d_{p_2}$.
Applying Corollary~\ref{0803-2}(i) to the vertex $u$, we obtain
\begin{equation} \label{0805-4}
\frac{1}{d_{p_1}} + \frac{1}{d_{p_2}}
= \frac{2}{4} + \frac{3 \cdot 2^2}{4m} = \frac{1}{2} + \frac{3}{m}.
\end{equation}
Since $d_{p_1} \leq d_{p_2}$, we have
\[ \frac{1}{2} + \frac{3}{m} = \frac{1}{d_{p_1}} + \frac{1}{d_{p_2}} \leq \frac{2}{d_{p_1}}, \]
and hence
\[ d_{p_1} \leq \frac{4m}{m+6} = 4 - \frac{24}{m+6} < 4. \]
From $\delta = 2$, it follows that $d_{p_1}$ is either $2$ or $3$.
\begin{adjustwidth}{1em}{1em}
\begin{claim}
We have $d_{p_1} = 2$.
\end{claim}
\begin{proof}[Proof of Claim]
\renewcommand{\qedsymbol}{$\blacklozenge$}
We now assume that $d_{p_1}=3$ and derive a contradiction.
Equality~\eqref{0805-4} becomes
\[ \frac{1}{3} + \frac{1}{d_{p_2}} = \frac{1}{2} + \frac{3}{m}, \]
which can be rewritten as $(d_{p_2} - 6)(m+18) = -108$.
Since $m+18 > 0$, we have $d_{p_2} - 6 < 0$.
Moreover, since $3 = d_{p_1} \leq d_{p_2}$, we have $d_{p_2} \in \{ 3, 4, 5 \}$.
Thus,
\begin{equation} \label{0806-1}
(d_{p_2}, m) = (3, 18), (4, 36), (5, 90).
\end{equation}
On the other hand, applying Corollary~\ref{0803-2}(ii) to the vertices $p_1$ and $p_2$, we obtain
\[ \frac{3 \cdot 3 \cdot d_{p_2}}{4m} = \sum_{z \in N(p_1) \cap N(p_2)} \frac{1}{d_z} \geq \frac{1}{d_u} = \frac{1}{2}, \]
that is,
\[ m \leq \frac{9}{2} d_{p_2}. \]
However, none of the solutions given in Equality~\eqref{0806-1} satisfies this inequality, a contradiction.
\end{proof}
\end{adjustwidth}
From $d_{p_1} = 2$ and Equality~\eqref{0805-4}, we obtain $d_{p_2} = m/3$.
Let $r$ be the neighbor of $p_1$ other than $u$.
Since $p_1$ and $p_2$ belong to the same partite set, they are not adjacent.
In particular, $r \neq p_2$.
The situation established so far is illustrated in Figure~\ref{0806-2}.

\begin{figure}[htb]
\centering
\begin{tikzpicture}
[scale = 0.5,
v/.style = {circle, fill = black, inner sep = 0.6mm},
u/.style = {circle, fill = white, inner sep = 0.0mm}
]

\node[u, label=below:$\TR{2}$] (d1) at (0, 0) {};
\node[u, label=right:$\TR{2}$] (d2) at (-2, -2) {};
\node[u, label=left:$\TR{\frac{m}{3}}$] (d3) at (2, -2) {};
\node[v, label=above:$u$] (1) at (0, 0) {};
\node[v, label=left:$p_1$] (2) at (-2, -2) {};
\node[v, label=right:$p_2$] (3) at (2, -2) {};
\node[v, label=left:$r$] (4) at (-2, -5) {};
\draw[line width = 1pt] (1) to (2);
\draw[line width = 1pt] (1) to (3);
\draw[line width = 1pt] (2) to (4);
\end{tikzpicture}
\caption{The local structure around $u$ at this stage. The red numbers indicate the degrees of the corresponding vertices.} \label{0806-2}
\end{figure}
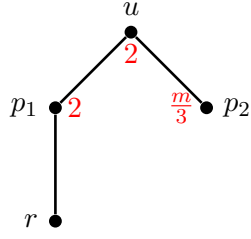

Applying Corollary~\ref{0803-2}(ii) to the vertices $p_1$ and $p_2$, we obtain
\[ \sum_{z \in N(p_1) \cap N(p_2)} \frac{1}{d_z} = \frac{3 \cdot 2 \cdot \frac{m}{3}}{4m} = \frac{1}{2}. \]
On the other hand, since
\[ \frac{1}{2} = \sum_{z \in N(p_1) \cap N(p_2)} \frac{1}{d_z}
\geq \frac{1}{d_u} = \frac{1}{2}, \]
we must have $N(p_1)\cap N(p_2)= \{ u \}$.
In particular, $r$ is not adjacent to $p_2$.
Let $V_0 = \{u\}, V_1 = \{p_1, p_2 \}, V_2, V_3, \dots$ be the distance partition with respect to $u$.
As $d_{p_1}=2$, every vertex $x \in V_2 \setminus \{r\}$ must be adjacent to $p_2$.
Since $d_{p_2} = \frac{m}{3}$, let $x_1, x_2, \dots, x_{\frac{m}{3}-1}$ be the neighbors of $p_2$ other than $u$.
Figure~\ref{0806-3} extends Figure~\ref{0806-2} by incorporating the additional information obtained so far.

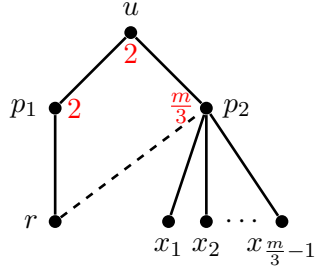
\begin{figure}[htb]
\centering
\begin{tikzpicture}
[scale = 0.5,
v/.style = {circle, fill = black, inner sep = 0.6mm},
u/.style = {circle, fill = white, inner sep = 0.0mm}
]
\node[u] (99) at (3, -5) {$\cdots$};
%
\node[u, label=below:$\TR{2}$] (d1) at (0, 0) {};
\node[u, label=right:$\TR{2}$] (d2) at (-2, -2) {};
\node[u, label=left:$\TR{\frac{m}{3}}$] (d3) at (2, -2) {};
\node[v, label=above:$u$] (1) at (0, 0) {};
\node[v, label=left:$p_1$] (2) at (-2, -2) {};
\node[v, label=right:$p_2$] (3) at (2, -2) {};
\node[v, label=left:$r$] (4) at (-2, -5) {};
\node[v, label=below:$x_1$] (5) at (1, -5) {};
\node[v, label=below:$x_2$] (6) at (2, -5) {};
\node[v, label=below:$x_{\frac{m}{3}-1}$] (7) at (4, -5) {};
\draw[line width = 1pt] (1) to (2);
\draw[line width = 1pt] (1) to (3);
\draw[line width = 1pt] (2) to (4);
\draw[line width = 1pt, dashed] (4) to (3);
\draw[line width = 1pt] (3) to (5);
\draw[line width = 1pt] (3) to (6);
\draw[line width = 1pt] (3) to (7);
\end{tikzpicture}
\caption{The structure determined at this stage.
The dashed line indicates that the two vertices are not adjacent.
The red numbers indicate the degrees of the corresponding vertices.} \label{0806-3}
\end{figure}

Applying Corollary~\ref{0803-2}(i) to the vertex $p_1$, we obtain
\[ \frac{1}{d_r} + \frac{1}{2} = \frac{2}{4} + \frac{3 \cdot 2^2}{4m}, \]
and hence $d_r=\frac{m}{3}$.
Since $d_r = \frac{m}{3}$, let $N(r) \setminus \{p_1\} = \{ y_1, y_2, \dots, y_{\frac{m}{3}-1} \}$.
For each $i \in \{ 1, 2, \dots, \frac{m}{3}-1 \}$, applying Corollary~\ref{0803-2}(ii) to the vertices $y_i$ and $p_1$, we obtain
\[ \frac{1}{\frac{m}{3}} = \frac{3 \cdot 2 \cdot d_{y_i}}{4m}, \]
and hence $d_{y_i} = 2$.
Similarly, for each $i \in \{ 1, 2, \dots, \frac{m}{3}-1 \}$, applying Corollary~\ref{0803-2}(ii) to the vertices $x_i$ and $u$ yields $d_{x_i} = 2$.
Figure~\ref{0807-1} extends Figure~\ref{0806-3} by incorporating the additional information obtained so far.

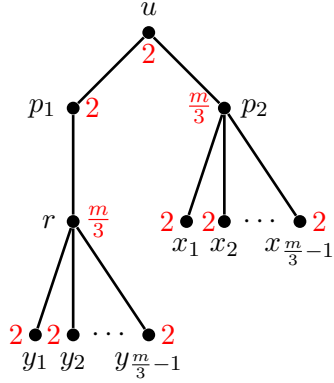
\begin{figure}[htb]
\centering
\begin{tikzpicture}
[scale = 0.5,
v/.style = {circle, fill = black, inner sep = 0.6mm},
u/.style = {circle, fill = white, inner sep = 0.0mm}
]
\node[u] (99) at (3, -5) {$\cdots$};
\node[u] (97) at (-1, -8) {$\cdots$};
\node[u, label=below:$\TR{2}$] (d1) at (0, 0) {};
\node[u, label=right:$\TR{2}$] (d2) at (-2, -2) {};
\node[u, label=left:$\TR{\frac{m}{3}}$] (d3) at (2, -2) {};
\node[u, label=right:$\TR{\frac{m}{3}}$] (d4) at (-2, -5) {};
\node[u, label=left:$\TR{2}$] (d5) at (1, -5) {};
\node[u, label=left:$\TR{2}$] (d6) at (2.1, -5) {};
\node[u, label=right:$\TR{2}$] (d7) at (4, -5) {};
\node[u, label=left:$\TR{2}$] (d8) at (-3, -8) {};
\node[u, label=left:$\TR{2}$] (d9) at (-2, -8) {};
\node[u, label=right:$\TR{2}$] (d10) at (0, -8) {};
\node[v, label=above:$u$] (1) at (0, 0) {};
\node[v, label=left:$p_1$] (2) at (-2, -2) {};
\node[v, label=right:$p_2$] (3) at (2, -2) {};
\node[v, label=left:$r$] (4) at (-2, -5) {};
\node[v, label=below:$x_1$] (5) at (1, -5) {};
\node[v, label=below:$x_2$] (6) at (2, -5) {};
\node[v, label=below:$x_{\frac{m}{3}-1}$] (7) at (4, -5) {};
\node[v, label=below:$y_1$] (8) at (-3, -8) {};
\node[v, label=below:$y_2$] (9) at (-2, -8) {};
\node[v, label=below:$y_{\frac{m}{3}-1}$] (10) at (0, -8) {};

\draw[line width = 1pt] (1) to (2);
\draw[line width = 1pt] (1) to (3);
\draw[line width = 1pt] (2) to (4);
\draw[line width = 1pt] (3) to (5);
\draw[line width = 1pt] (3) to (6);
\draw[line width = 1pt] (3) to (7);
\draw[line width = 1pt] (4) to (8);
\draw[line width = 1pt] (4) to (9);
\draw[line width = 1pt] (4) to (10);
\end{tikzpicture}
\caption{The structure determined at this stage.
The red numbers indicate the degrees of the corresponding vertices.} \label{0807-1}
\end{figure}

Here, computing the sum of the degrees of the vertices considered so far, we obtain
\[ d_u + d_{p_1} + d_{p_2} + d_r + \sum_{i=1}^{\frac{m}{3}-1} d_{x_i} + \sum_{i=1}^{\frac{m}{3}-1} d_{y_i} = 2m. \]
Hence, by the handshaking lemma, the graph $G$ has no other vertices.
On the other hand, since $d_{x_1} = 2$, the vertex $x_1$ must be adjacent to one of the vertices $y_1, \dots, y_{\frac{m}{3}-1}$.
Without loss of generality, we may assume that $x_1$ is adjacent to $y_1$.
Similarly, we may assume that $x_2$ is adjacent to $y_2$, and repeating this argument shows that $x_i$ is adjacent to $y_i$ for each $i \in \{1, 2, \dots, \frac{m}{3}-1 \}$.
As shown in Figure~\ref{0807-2}, the graph $G$ is isomorphic to the uniform theta graph $\Theta(\frac{m}{3}, 2)$.
Comparing the numbers of vertices, we obtain $2t + 2 = 4+ 2(\frac{m}{3} - 1)$, that is, $\frac{m}{3} = t$, which proves the desired isomorphism.
\end{proof}

\begin{figure}[htb]
\centering
\begin{tikzpicture}
[scale = 0.5,
v/.style = {circle, fill = black, inner sep = 0.6mm},
u/.style = {circle, fill = white, inner sep = 0.0mm}
]
\node[u] (99) at (3, -5) {$\cdots$};
\node[u] (97) at (-1, -8) {$\cdots$};
\node[v, label=above:$u$] (1) at (0, 0) {};
\node[v, label=left:$p_1$] (2) at (-2, -2) {};
\node[v, label=right:$p_2$] (3) at (2, -2) {};
\node[v, label=left:$r$] (4) at (-2, -5) {};
\node[v, label=left:$x_1$] (5) at (1, -5) {};
\node[v, label=below:$x_2$] (6) at (2, -5) {};
\node[v, label=right:$x_{\frac{m}{3}-1}$] (7) at (4, -5) {};
\node[v, label=below:$y_1$] (8) at (-3, -8) {};
\node[v, label=below:$y_2$] (9) at (-2, -8) {};
\node[v, label=below:$y_{\frac{m}{3}-1}$] (10) at (0, -8) {};
%
\draw[line width = 1pt] (1) to (2);
\draw[line width = 1pt] (1) to (3);
\draw[line width = 1pt] (2) to (4);
\draw[line width = 1pt] (3) to (5);
\draw[line width = 1pt] (3) to (6);
\draw[line width = 1pt] (3) to (7);
\draw[line width = 1pt] (4) to (8);
\draw[line width = 1pt] (4) to (9);
\draw[line width = 1pt] (4) to (10);
\draw[line width = 1pt] (5) to (8);
\draw[line width = 1pt] (6) to (9);
\draw[line width = 1pt] (7) to (10);
\end{tikzpicture}
\raisebox{63pt}{$\quad = \quad$}
\begin{tikzpicture}
[scale = 0.5,
v/.style = {circle, fill = black, inner sep = 0.6mm},
u/.style = {circle, fill = white, inner sep = 0.0mm}
]
\node[u] (99) at (-2, -1.8) {$\vdots$};
\node[u] (98) at (2, -1.8) {$\vdots$};
\node[v, label=left:$r$] (1) at (-5, 0) {};
\node[v, label=above:$p_1$] (2) at (-2, 3.5) {};
\node[v, label=above:$y_1$] (3) at (-2, 1.5) {};
\node[v, label=below:$y_{\frac{m}{3}-1}$] (4) at (-2, -3.5) {};
\node[v, label=above:$u$] (5) at (2, 3.5) {};
\node[v, label=above:$x_1$] (6) at (2, 1.5) {};
\node[v, label=below:$x_{\frac{m}{3}-1}$] (7) at (2, -3.5) {};
\node[v, label=right:$p_2$] (8) at (5, 0) {};
\node[v, label=above:$y_2$] (9) at (-2, -0.5) {};
\node[v, label=above:$x_2$] (10) at (2, -0.5) {};
\draw[line width = 1pt] (1) to (2);
\draw[line width = 1pt] (1) to (3);
\draw[line width = 1pt] (1) to (4);
\draw[line width = 1pt] (5) to (2);
\draw[line width = 1pt] (6) to (3);
\draw[line width = 1pt] (7) to (4);
\draw[line width = 1pt] (8) to (5);
\draw[line width = 1pt] (8) to (6);
\draw[line width = 1pt] (8) to (7);
\draw[line width = 1pt] (1) to (9);
\draw[line width = 1pt] (8) to (10);
\draw[line width = 1pt] (9) to (10);
\end{tikzpicture}
\caption{The resulting structure of $G$, which shows that the graph $G$ is isomorphic to the uniform theta graph $\Theta(\frac{m}{3}, 2) = \Theta(t, 2)$.} \label{0807-2}
\end{figure}
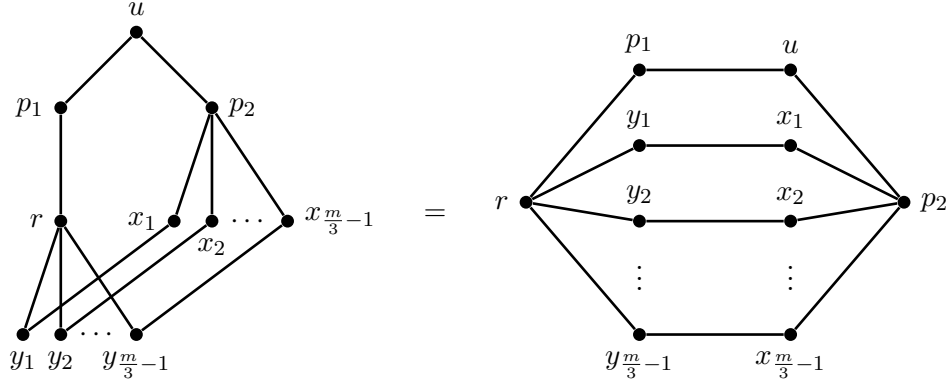

To highlight the connection between Theorem~\ref{0903-2} and spectral characterization with respect to the normalized Laplacian matrix, we conclude this paper by presenting the following reformulation of Theorem~\ref{0903-2}.

\begin{cor}
Let $t$ be a positive integer.
The uniform theta graph $\Theta(t,2)$ is determined by its normalized Laplacian spectrum.
\end{cor}

\section*{Acknowledgements}
S.K. is supported by JSPS KAKENHI (Grant Number JP24K16970).

\section*{Data availability}
This article has no associated data.

\bibliographystyle{plain}
\bibliography{mybibs}

@article{banerjee2017normalized,
  title={On the normalized spectrum of threshold graphs},
  author={Banerjee, A. and Mehatari, R.},
  journal={Linear Algebra and its Applications},
  volume={530},
  pages={288--304},
  year={2017},
  publisher={Elsevier}
}

@article{berman2018family,
  title={A family of graphs that are determined by their normalized {L}aplacian spectra},
  author={Berman, A. and Chen, D. and Chen, Z. and Liang, W. and Zhang, X.},
  journal={Linear Algebra and its Applications},
  volume={548},
  pages={66--76},
  year={2018},
  publisher={Elsevier}
}

@article{bhakta2024grover,
  title={Grover walks on unitary Cayley graphs and integral regular graphs},
  author={Bhakta, K. and Bhattacharjya, B.},
  journal={arXiv preprint arXiv:2405.01020},
  year={2024}
}

@article{bhakta2025periodicity,
  title={Periodicity and perfect state transfer of {G}rover walks on quadratic unitary {C}ayley graphs},
  author={Bhakta, K. and Bhattacharjya, B.},
  journal={Quantum Information Processing},
  volume={24},
  number={8},
  pages={260},
  year={2025},
  publisher={Springer}
}

@article{bhakta2026state,
  title={State transfer in {G}rover walks on unitary and quadratic unitary {C}ayley graphs over finite commutative rings},
  author={Bhakta, K. and Bhattacharjya, B.},
  journal={Discrete Mathematics},
  volume={349},
  number={8},
  pages={115151},
  year={2026},
  publisher={Elsevier}
}

@article{braga2015trees,
  title={Trees with $4$ or $5$ distinct normalized {L}aplacian eigenvalues},
  author={Braga, R. O. and Del-Vecchio, R. R. and Rodrigues, V. M. and Trevisan, V.},
  journal={Linear Algebra and its Applications},
  volume={471},
  pages={615--635},
  year={2015},
  publisher={Elsevier}
}

@article{butler2011construction,
  title={A construction of cospectral graphs for the normalized {L}aplacian},
  author={Butler, S. and Grout, J.},
  journal={The Electronic Journal of Combinatorics},
  volume={18},
  number={1},
  pages={P231},
  year={2011},
  doi={10.37236/718}
}

@article{van2003graphs,
  title={Which graphs are determined by their spectrum?},
  author={Dam, E. R. van and Haemers, W. H.},
  journal={Linear Algebra and its applications},
  volume={373},
  pages={241--272},
  year={2003},
  publisher={Elsevier}
}

@article{van2009developments,
  title={Developments on spectral characterizations of graphs},
  author={Dam, E. R. van and Haemers, W. H.},
  journal={Discrete Mathematics},
  volume={309},
  number={3},
  pages={576--586},
  year={2009},
  publisher={Elsevier}
}

@article{van2011graphs,
  title={Graphs whose normalized {L}aplacian has three eigenvalues},
  author={Dam, E. R. van and Omidi, G. R.},
  journal={Linear algebra and its applications},
  volume={435},
  number={10},
  pages={2560--2569},
  year={2011},
  publisher={Elsevier}
}

@book{godsil2013algebraic,
  title={Algebraic graph theory},
  author={Godsil, C. and Royle, G.},
  volume={207},
  year={2013},
  publisher={Springer Science \& Business Media}
}

@article{gunthard1956zusammenhang,
  title={{Zusammenhang von Graphentheorie und MO-Theorie von Molekeln mit Systemen konjugierter Bindungen}},
  author={G{\"u}nthard, H. H. and Primas, H.},
  journal={Helvetica Chimica Acta},
  volume={39},
  number={6},
  pages={1645--1653},
  year={1956},
  doi={10.1002/hlca.19560390623}
}

@article{higuchi2014spectral,
  title={Spectral and asymptotic properties of {G}rover walks on crystal lattices},
  author={Higuchi, Y. and Konno, N. and Sato, I. and Segawa, E.},
  journal={Journal of Functional Analysis},
  volume={267},
  number={11},
  pages={4197--4235},
  year={2014},
  publisher={Elsevier}
}

@article{higuchi2017periodicity,
  title={Periodicity of the discrete-time quantum walk on a finite graph},
  author={Higuchi, Y. and Konno, N. and Sato, I. and Segawa, E.},
  journal={Interdisciplinary Information Sciences},
  volume={23},
  number={1},
  pages={75--86},
  year={2017},
  publisher={The Editorial Committee of the Interdisciplinary Information Sciences}
}

@article{hoffman1963polynomial,
  title={On the polynomial of a graph},
  author={Hoffman, A. J.},
  journal={The American Mathematical Monthly},
  volume={70},
  number={1},
  pages={30--36},
  year={1963},
  publisher={Taylor \& Francis}
}

@article{huang2019graphs,
  title={On graphs with three or four distinct normalized {L}aplacian eigenvalues},
  author={Huang, X. and Huang, Q.},
  journal={Algebra Colloquium},
  volume={26},
  number={01},
  pages={65--82},
  year={2019}
}

@article{jost2023petals,
  title={Petals and books: {T}he largest {L}aplacian spectral gap from 1},
  author={Jost, J. and Mulas, R. and Zhang, D.},
  journal={Journal of Graph Theory},
  volume={104},
  number={4},
  pages={727--756},
  year={2023},
  publisher={Wiley Online Library}
}

@article{kubota2018generalizedBT,
  title={Periodicity of {G}rover walks on generalized {B}ethe trees},
  author={Kubota, S. and Segawa, E. and Taniguchi, T. and Yoshie, Y.},
  journal={Linear Algebra and its Applications},
  volume={554},
  pages={371--391},
  year={2018},
  publisher={Elsevier}
}

@article{kubota2021periodicity,
  title={Periodicity of quantum walks defined by mixed paths and mixed cycles},
  author={Kubota, S. and Sekido, H. and Yata, H.},
  journal={Linear Algebra and its Applications},
  volume={630},
  pages={15--38},
  year={2021},
  publisher={Elsevier}
}

@article{kubota2021quantum,
  title={Quantum walks defined by digraphs and generalized {H}ermitian adjacency matrices},
  author={Kubota, S. and Segawa, E. and Taniguchi, T.},
  journal={Quantum Information Processing},
  volume={20},
  number={3},
  pages={95},
  year={2021},
  publisher={Springer}
}

@article{kubota2022bipartite,
  title={Periodicity of {G}rover walks on bipartite regular graphs with at most five distinct eigenvalues},
  author={Kubota, S.},
  journal={Linear Algebra and its Applications},
  volume={654},
  pages={125--142},
  year={2022},
  publisher={Elsevier}
}

@article{kubota2025regular,
  title={Regular graphs to induce even periodic {G}rover walks},
  author={Kubota, S. and Sekido, H. and Yoshino, K.},
  journal={Discrete Mathematics},
  volume={348},
  number={3},
  pages={114345},
  year={2025},
  publisher={Elsevier}
}

@article{kubota2026entanglement,
  title={Entanglement entropy in two-particle {G}rover walks on graphs},
  author={Kubota, S. and Matsubara, H. and Segawa, E.},
  journal={arXiv preprint arXiv:2607.09066},
  year={2026}
}

@article{li2014bounds,
  title={Bounds on normalized {L}aplacian eigenvalues of graphs},
  author={Li, J. and Guo, J. and Shiu, W. C.},
  journal={Journal of Inequalities and Applications},
  volume={2014},
  number={1},
  pages={316},
  year={2014},
  publisher={Springer}
}

@article{li2020split,
  title={On split graphs with three or four distinct (normalized) {L}aplacian eigenvalues},
  author={Li, S. and Sun, W.},
  journal={Journal of Combinatorial Designs},
  volume={28},
  number={11},
  pages={763--782},
  year={2020},
  publisher={Wiley Online Library}
}

@article{panda2021order,
  title={Order from chaos in quantum walks on cyclic graphs},
  author={Panda, A. and Benjamin, C.},
  journal={Physical Review A},
  volume={104},
  number={1},
  pages={012204},
  year={2021},
  publisher={APS}
}

@article{rath2026quantum,
  title={Quantum cryptography compatible with noisy intermediate-scale quantum devices based on {P}arrondo dynamics in discrete-time quantum walks},
  author={Rath, A. and Panda, D. K. and Benjamin, C.},
  journal={Physical Review A},
  volume={114},
  number={1},
  pages={012615},
  year={2026},
  publisher={APS}
}

@article{sason2025spectral,
  title={On spectral graph determination},
  author={Sason, I. and Krupnik, N. and Hamud, S. and Berman, A.},
  journal={Mathematics},
  volume={13},
  number={4},
  pages={549},
  year={2025},
  publisher={MDPI}
}

@article{sun2019second,
  title={On the second largest normalized {L}aplacian eigenvalue of graphs},
  author={Sun, S. and Das, K. Ch.},
  journal={Applied Mathematics and Computation},
  volume={348},
  pages={531--541},
  year={2019},
  publisher={Elsevier}
}

@article{sun2020normalized,
  title={Normalized {L}aplacian spectrum of complete multipartite graphs},
  author={Sun, S. and Das, K. C.},
  journal={Discrete Applied Mathematics},
  volume={284},
  pages={234--245},
  year={2020},
  publisher={Elsevier}
}

@article{sun2024characterization,
  title={Characterization of Graphs with Some Normalized {L}aplacian Eigenvalue Having Multiplicity $n-4$},
  author={Sun, S.},
  journal={Communications in Mathematics and Statistics},
  pages={1--11},
  year={2024},
  publisher={Springer}
}

@article{sun2026spectral,
  title={On the spectral characterization of graphs with respect to the normalized {L}aplacian},
  author={Sun, S. and Sun, X. and Das, K. C.},
  journal={Discrete Mathematics},
  volume={349},
  number={6},
  pages={115034},
  year={2026},
  publisher={Elsevier}
}

@article{tian2021full,
  title={Full characterization of graphs having certain normalized {L}aplacian eigenvalue of multiplicity $n-3$},
  author={Tian, F. and Wang, Y.},
  journal={Linear Algebra and its Applications},
  volume={630},
  pages={69--83},
  year={2021},
  publisher={Elsevier}
}

@article{yoshie2017characterizations,
  title={Characterizations of graphs to induce periodic {G}rover walk},
  author={Yoshie, Y.},
  journal={Yokohama Mathematical Journal},
  volume={63},
  pages={9--23},
  year={2017}
}

@article{yoshie2019periodicities,
  title={Periodicities of {G}rover walks on distance-regular graphs},
  author={Yoshie, Y.},
  journal={Graphs and Combinatorics},
  volume={35},
  number={6},
  pages={1305--1321},
  year={2019},
  publisher={Springer}
}

@article{yoshie2023odd,
  title={Odd-periodic {G}rover Walks},
  author={Yoshie, Y.},
  journal={Quantum Information Processing},
  volume={22},
  number={8},
  pages={316},
  year={2023},
  publisher={Springer}
}

@article{zhao2025solving,
  title={Solving the problem about the second largest normalized {L}aplacian eigenvalue},
  author={Zhao, X. and You, L.},
  journal={Discrete Applied Mathematics},
  volume={367},
  pages={8--21},
  year={2025},
  publisher={Elsevier}
}

\end{document}